\documentclass[12pt,a4paper]{article}

\usepackage{amsfonts,amsmath,amssymb,amsthm,amsopn}
\usepackage{graphicx}
\usepackage{algorithmic}
\usepackage[numbers]{natbib}
\usepackage{hyperref}
\hypersetup{hidelinks}
\usepackage{bm,color}
\usepackage{stmaryrd}
\usepackage[T1]{fontenc}
\usepackage{float, caption, subcaption}
\usepackage{booktabs}
\usepackage{multirow}
\usepackage{siunitx}
\usepackage{enumitem}
\usepackage{lscape}
\usepackage[section]{placeins}

\newcommand\bF{\bm{F}}
\newcommand\bq{\bm{q}}

\newcommand\bu{\bm{u}}
\newcommand\bU{\bm{U}}
\newcommand\bx{\bm{x}}

\newcommand\bK{\bm{K}}
\newcommand\bW{\bm{W}}
\newcommand\bZ{\bm{Z}}
\newcommand\bS{\bm{S}}

\newcommand\btau{\bm{\tau}}

\newcommand\dd{\mathrm{d}}
\newcommand\pd[2]{\frac{\partial {#1}}{\partial {#2}}}

\newcommand\abs[1]{\lvert #1 \rvert}
\newcommand\norm[1]{\left\lVert #1 \right\rVert}

\newcommand\xr{i+\frac12}
\newcommand\xl{i-\frac12}
\newcommand\yr{j+\frac12}
\newcommand\yl{j-\frac12}

\allowdisplaybreaks
\theoremstyle{plain}\newtheorem{definition}{Definition}[section]

\newtheorem{theorem}{Theorem}[section]

\theoremstyle{definition}
\newtheorem{example}{Example}[section]
\newtheorem{proposition}{Proposition}[section]
\newtheorem{remark}{Remark}[section]
\newtheorem{lemma}{Lemma}[section]

\title{A compact high-order and positivity-preserving active flux method for compressible Navier--Stokes equations on Cartesian meshes}
\author{
Junming Duan\thanks{Corresponding author. School of Science and Engineering, The Chinese University of Hong Kong (Shenzhen), Guangdong, 518172, P.R. China. Email: \href{mailto:duanjunming@cuhk.edu.cn}{\texttt{duanjunming@cuhk.edu.cn}}.} \and
Wasilij Barsukow\thanks{Institut de Math\'ematiques de Bordeaux (IMB), CNRS UMR 5251, University of Bordeaux, Talence, 33405, France. Email: \href{mailto:wasilij.barsukow@math.u-bordeaux.fr}{\texttt{wasilij.barsukow@math.u-bordeaux.fr}}.} \and
Christian Klingenberg\thanks{Institute of Mathematics, University of W\"urzburg, Emil-Fischer-Stra\ss e 40, W\"urzburg, 97074, Germany. Email: \href{mailto:christian.klingenberg@uni-wuerzburg.de}{\texttt{christian.klingenberg@uni-wuerzburg.de}}.}
}
\date{}

\begin{document}
\maketitle

\begin{abstract}
This paper develops a compact fourth-order positivity-preserving active flux (AF) method for the one- and two-dimensional compressible Navier--Stokes equations on Cartesian meshes.
	The method retains the cell averages and shared point values of the standard third-order AF method as its degrees of freedom.
To avoid the order reduction that can arise when diffusion is discretized using operators from the standard third-order AF method, while maintaining compactness, the divergence of the viscous flux is discretized directly using compact fourth-order operators.
	For the inviscid part, incorporating a downwind point value into the biased stencil yields fourth-order accuracy.
	A monolithic flux limiting blends high-order total numerical fluxes with low-order positivity-preserving counterparts, treating the inviscid and viscous fluxes jointly while maintaining local conservation.
	Together with a scaling limiter for point values, this procedure preserves density and pressure positivity for both cell averages and point values.
	Numerical experiments demonstrate fourth-order convergence, positivity preservation, and accurate resolution of shocks and viscous flow structures.
	For the two-dimensional viscous shock tube, a comparison with a discontinuous Galerkin method shows improved computational efficiency in resolving complex interaction of shock waves and boundary layers.
\end{abstract}

\noindent\textbf{Keywords: }Active flux;  Navier--Stokes equations;  high-order accuracy;  finite volume method;  compact scheme;  positivity-preserving

\section{Introduction}\label{sec:introduction}
Computational fluid dynamics is central to weather prediction, pollutant transport, biomedical applications, and aerodynamic design \cite{Slotnick_2014_CFD},
with accurate modeling of viscous effects essential for resolving boundary layers and predicting aerodynamic drag.
The development of high-order numerical methods,
which can provide enhanced accuracy and computational efficiency \cite{Wang_2013_High‐order_IJNMF},
is therefore of fundamental importance in both scientific research and engineering practice.
In this work, we focus on the development of a compact high-order numerical method for solving the $d$-dimensional compressible Navier--Stokes (NS) equations
\begin{equation}\label{eq:ns}
	\partial_t \bU + \sum_{\ell=1}^{d} \partial_{x_\ell}
	\left(\bF_{\ell}^c(\bU)-\bF_{\ell}^v(\bU,\nabla\bU)\right)=\bm 0,
\end{equation}
where $\bU = (\rho, \rho\bu, E)^\top$ denotes the vector of conservative variables,
and $\bF_{\ell}^c$ and $\bF_{\ell}^v$ denote, respectively, the inviscid and viscous fluxes in the $x_\ell$-direction,
\begin{equation}\label{eq:ns_directional_fluxes}
	\bF_\ell^c=
	\begin{pmatrix}
		\rho u_\ell \\
		\rho\bu u_\ell+p\bm e_\ell \\
		(E+p)u_\ell
	\end{pmatrix},
	\quad
	\bF_\ell^v=
	\begin{pmatrix}
		0\\
		\btau\bm e_\ell \\
		\bu\mathbin{\cdot}(\btau\bm e_\ell)-q_\ell
	\end{pmatrix},
	\quad q_\ell=\bq\mathbin{\cdot}\bm e_\ell.
\end{equation}
Here, $\rho$ is the density, $\bu = (u_1, \dots, u_d)^\top$ is the velocity vector, $p$ is the pressure, and $E = \rho e+\frac12\rho\norm{\bu}^2$ is the total energy, with $e$ the specific internal energy.
The vector $\bm e_\ell$ is the unit vector with all zeros except for the $\ell$th component.
For a Newtonian fluid satisfying Fourier's law of heat conduction,
the viscous stress tensor and Fourier heat flux are given as follows,
\begin{equation}\label{eq:newton_fourier}
	\btau=\mu\left(\nabla\bu+\nabla\bu^{\top}
	-\frac{2}{3}(\nabla\mathbin{\cdot}\bu)\bm I_d\right),
	\quad
	\bq=-\kappa\nabla T,
\end{equation}
with $\mu>0$ the constant dynamic viscosity.
In the nondimensional form, the temperature is $T = p/\rho$ and the thermal conductivity is $\kappa=\frac{\gamma\mu}{(\gamma-1)\Pr}$,
where $\Pr>0$ is the Prandtl number,
which characterizes the relative importance of momentum and thermal diffusion.
To close the system \eqref{eq:ns}, the ideal-gas equation of state $p=(\gamma-1)\rho e$ is used,
with $\gamma>1$ the ratio of specific heats.

A variety of high-order accurate numerical methods have been developed for the compressible NS equations,
including compact finite difference methods \cite{Lele_1992_Compact_JCP},
$P_NP_M$ methods \cite{Dumbser_2010_Arbitrary_CF},
discontinuous Galerkin (DG) methods \cite{Bassi_1997_high_JCP},
and flux reconstruction (FR) methods \cite{Huynh_2007_Flux_InProceedings,Vincent_2011_New_JSC},
see also the review \cite{Wang_2013_High‐order_IJNMF}.
High-order methods can provide high resolution for smooth and multiscale flow structures on relatively coarse meshes.
For large-scale simulations,
the locality and compactness of the discretization are also important,
since localized data dependencies reduce communication and facilitate parallel implementation on modern high-performance computing architectures,
which is one of the reasons why DG methods are attractive.
However,
the number of local degrees of freedom (DoFs) in DG methods grows rapidly with the polynomial degree.
For example,
a tensor-product polynomial space of degree $k$ in each coordinate direction contains $(k+1)^d$ coefficients per conservative variable in $d$ dimensions,
leading to quadratic and cubic growth in memory requirements and computational costs in 2D and 3D, respectively.
These considerations motivate the development of compact high-order methods with fewer DoFs.

The active flux (AF) method provides a compact finite volume framework,
with fewer DoFs than DG methods.
Although its number of DoFs still grows quadratically (cubically) with the polynomial degree,
the comparison in \cite{Barsukow_2026_equivalence} shows higher computational efficiency.
Originally developed for hyperbolic conservation laws
\cite{Eymann_2011_Active_InProceedings,Eymann_2011_Active_InProceedingsa,Eymann_2013_Multidimensional_InProceedings},
inspired by \cite{VanLeer_1977_Towards_JCP},
AF methods evolve cell averages together with point values shared by neighboring cells,
which reduces the number of DoFs.
The cell average follows a conservative finite volume update,
and the point value provides additional interface information for high-order approximations.
As described below, the AF method is local because point values are evolved using compact finite difference operators, without quadrature or inverting a mass matrix.	
All these factors contribute to the efficiency of the AF methods.

In the original fully-discrete formulation \cite{Eymann_2011_Active_InProceedings,Eymann_2011_Active_InProceedingsa,Eymann_2013_Multidimensional_InProceedings},
the point values are evolved by exact or approximate evolution operators,
which are difficult to construct for general cases,
especially to achieve high-order accuracy \cite{Helzel_2019_new_JSC,Chudzik_2024_Active_JSC,Barsukow_2026_active_JSC,Chudzik_2026_fully_SJSC}.
A method-of-lines formulation was subsequently introduced in
\cite{Abgrall_2023_combination_CAMC},
where the evolution equations for both cell averages and point values are first discretized in space and then integrated in time using Runge--Kutta (RK) methods.
Upwinding is incorporated in the point value update using Jacobian splitting \cite{Abgrall_2023_combination_CAMC} or flux-vector splitting \cite{Duan_2025_Active_SJSC},
which has been extended to higher-order accuracy
\cite{Abgrall_2023_Extensions_EMMNA,Barsukow_2026_generalized_CF}, polygonal meshes \cite{Abgrall_2026_Virtual_JCP}, and magnetohydrodynamics \cite{Duan_2025_Active,Liu_2026_active_JCP}.

Compared with the hyperbolic case, AF discretizations of diffusion terms are relatively less developed.
One approach reformulates the diffusion terms as a first-order hyperbolic or relaxation system with auxiliary variables
\cite{Nishikawa_2016_Third_CF,Duan_2025_asymptotic}.
The pseudo-time formulation in \cite{Nishikawa_2016_Third_CF} uses inner iterations, whereas \cite{Duan_2025_asymptotic} uses a diagonally implicit RK method.
Another discretizes the diffusion terms directly using auxiliary gradient approximations \cite{He_2021_Towards}, but it is second-order accurate.
In \cite{Duan_2026_compact}, a compact fourth-order AF method is constructed for solving convection--diffusion equations with degenerate diffusion coefficients,
where the diffusion term is discretized based on a rewriting into a doubly conservative form of the diffusion potential.

A further challenge is to preserve physically admissible states in the presence of strong shocks or near-vacuum states,
where high-order methods are prone to producing negative density or pressure due to unphysical oscillations.
This issue has been considered in the AF methods  \cite{Chudzik_2021_Cartesian_AMC,Duan_2025_Active_SJSC,Abgrall_2025_Bound_CCP,Abgrall_2025_novel_SJSC,Abgrall_2026_Robust_JCP,Liu_2026_active_JCP}
and other methods for the NS equations \cite{Fan_2021_Positivity_JCP,Fan_2022_Positivity_JCP,Lin_2023_positivity_JCP,Yamaleev_2023_High_JSC}, to name a few.
For the compressible NS equations, the viscous flux depends on velocity and temperature gradients and contains coupled stress components in multiple dimensions, making positivity-preserving (PP) limiting more delicate.

This work develops a compact fourth-order PP AF method for the one- and two-dimensional compressible NS equations on Cartesian meshes.
The choice of fourth-order operators is motivated by the order reduction to second-order accuracy in two dimensions,
as analyzed and numerically verified in \cite{Duan_2025_asymptotic} for the AF method using the standard third-order operators.
This paper discretizes the viscous flux divergence directly using compact fourth-order operators, while retaining the DoFs of the standard third-order AF method.
For the inviscid part, a biased stencil incorporating a downwind point value achieves fourth-order accuracy with the same DoFs as the standard third-order AF method.
To preserve density and pressure positivity, we construct a monolithic convex limiter for the numerical flux.
The high-order inviscid and viscous contributions are limited jointly against a low-order PP flux based on the local Lax--Friedrichs (LLF) scheme in \cite{Zhang_2017_positivity_JCP}.
A decomposition into directional intermediate states yields a convex representation of the cell average update and 
a sufficient CFL condition for positivity preservation in the explicit strong-stability-preserving (SSP)-RK method.
A scaling limiter enforces the PP property for point values.
Numerical experiments assess the accuracy and PP property of the proposed method.
Its computational efficiency is examined through a comparison with the DG method \cite{Lin_2023_positivity_JCP} for the viscous shock tube.
In this test, AF is approximately ten times faster at a comparable level of accuracy and with the same number of DoFs.

The remainder of this paper is organized as follows.
Section~\ref{sec:compact_AF} presents the 1D and 2D compact AF discretizations.
Section~\ref{sec:ns_pp} presents the first-order PP LLF scheme and the PP limiting.
Numerical experiments are conducted in Section~\ref{sec:results},
followed by concluding remarks in Section~\ref{sec:conclusion}.

\section{The compact AF formulation}\label{sec:compact_AF}
\subsection{The 1D formulation}\label{sec:1d_formulation}
The 1D NS equations can be written as
\begin{equation}\label{eq:ns_1d}
	\partial_t\bU + \partial_x\bF^c(\bU) = \partial_x\bF^v(\bU, \partial_x \bZ) = \partial_x(\bK(\bW) \partial_x\bZ),
\end{equation}
where
\begin{equation*}
	\bF^c=
	\begin{pmatrix}
		\rho u\\ \rho u^2+p\\ (E+p)u
	\end{pmatrix},\quad
	\bK(\bW)=
	\begin{pmatrix}
		0&0&0\\
		0&\frac43\mu&0\\
		0&\frac43\mu u&\frac{\gamma\mu}{(\gamma-1)\Pr}
	\end{pmatrix},
\end{equation*}
with
\begin{equation}\label{eq:ns_1d_variables}
	\bW=(\rho,u,p)^{\top},\quad
	\bZ=\left(0,u,T\right)^{\top}.
\end{equation}

Assume that a 1D computational domain is divided into $N$ uniform cells
$I_i = [x_{\xl}, x_{\xr}]$ with the cell size $\Delta x = x_{\xr}-x_{\xl}$, $i=1,\cdots,N$.
The cell centers are denoted by $x_i = (x_{\xl} + x_{\xr})/2$.
The DoFs are the cell averages and shared interface point values of conservative variables,
\begin{equation*}
	\overline{\bU}_i(t)=\frac{1}{\Delta x}\int_{I_i}\bU_h(x,t)\,\dd x,
	\quad
	\bU_{i+1/2}(t)=\bU_h(x_{i+\frac12},t),
\end{equation*}
with $\bU_h(x,t)$ the numerical solution.
The AF method for solving \eqref{eq:ns_1d} can be constructed as follows,
\begin{subequations}\label{eq:1d_semi_af}
	\begin{align}
		&\frac{\dd \overline{\bU}_i }{\dd t} + \frac{1}{\Delta x}\left( \bF^c_{\xr} - \bF^c_{\xl} \right) = \frac{1}{\Delta x}\left( \bK(\bW_{\xr})(\bZ_x)_{\xr} - \bK(\bW_{\xl})(\bZ_x)_{\xl} \right), \label{eq:af_av_1d}\\
		&\frac{\dd \bU_{\xr} }{\dd t} + \mathcal{D}_{+} (\bF^{c,+})_{\xr} + \mathcal{D}_{-}(\bF^{c,-})_{\xr} = \mathcal{D}^{2} (\bK(\bW), \bZ)_{\xr}, \label{eq:af_pnt_1d}\\
		&(\bZ_x)_{\xr} = \mathcal{D}_{c}(\bZ)_{\xr}, \label{eq:z_derivative_1d}
	\end{align}
\end{subequations}
where the cell average update \eqref{eq:af_av_1d} follows the standard finite volume method,
which ensures conservation.
The corresponding $\bW_{i}$, $\bZ_{i}$, and $\bF^{c,\pm}_i = \bF^{c,\pm}(\bU_i)$ are then evaluated from $\bU_i $.
In the point value update \eqref{eq:af_pnt_1d}, the inviscid part uses the flux vector splitting introduced in \cite{Duan_2025_Active_SJSC},
and the viscous part is based on a compact discretization.
To be specific, the inviscid flux is split using the local Lax--Friedrichs flux vector splitting (FVS), given by
\begin{equation*}
	\bF^{c,\pm} = \frac{1}{2}\left(\bF^c(\bU) \pm \alpha_{\xr} \bU\right),
	\quad
	\lambda\left(\pd{\bF^{c,+}}{\bU}\right) \geqslant 0,
	\quad
	\lambda\left(\pd{\bF^{c,-}}{\bU}\right) \leqslant 0,
\end{equation*}
where $\lambda$ denotes the eigenvalue,
and the numerical viscosity coefficient $\alpha_{\xr}$ is fixed across the spatial stencil and determined by
\begin{equation*}
	\alpha_{\xr} = \max\{ \varrho(\bU_{\xl}), ~\varrho(\bU_{i}), ~\varrho(\bU_{\xr}), ~\varrho(\bU_{i+1}), ~\varrho(\bU_{i+\frac32}) \},
\end{equation*}
with $\varrho$ the spectral radius of the Jacobian matrix $\partial\bF^c/\partial\bU$.
Overbars denote cell averages, whereas quantities without overbars are point values. The cell-centered conservative variables are reconstructed through Simpson's rule
\begin{equation*}
	\bU_i = \frac{1}{4}( 6\overline{\bU}_i - (\bU_{\xl} + \bU_{\xr}) ),
\end{equation*}
which is a $4$th-order approximation for smooth solutions.
Each component of the split flux is discretized using the following biased operators.
Here $z$ denotes a generic scalar component of the quantity being differentiated, distinct from the vector $\bZ$.
Note that $z_i$ is a cell-center point value, not a cell average.
The operators are obtained by differentiating the corresponding cubic reconstructions based on $\{z_{i-\frac12}, z_{i}, z_{i+\frac12}, z_{i+1}\}$ or $\{z_{i}, z_{i+\frac12}, z_{i+1}, z_{i+\frac32}\}$ at $x_{\xr}$,
\begin{subequations}\label{eq:1d_upwind_ops}
	\begin{align}
		\mathcal{D}_{+}(z)_{\xr} &= \frac{1}{3\Delta x}\left( z_{\xl} - 6z_{i} + 3z_{\xr} + 2z_{i+1} \right), \label{eq:1d_upwind_ops_4p} \\
		\mathcal{D}_{-}(z)_{\xr} &= \frac{1}{3\Delta x}\left( -2z_{i} - 3z_{\xr} + 6z_{i+1} - z_{i+\frac32} \right). \label{eq:1d_upwind_ops_4m}
	\end{align}
\end{subequations}
Biased stencils provide upwinding for the split inviscid fluxes, while their symmetric average gives a central difference approximation for the viscous gradients as there is no upwinding mechanism for the diffusion part.
The first-order derivative in \eqref{eq:z_derivative_1d} is obtained by using the central operator defined by
\begin{equation}\label{eq:1d_diff_central_4th}
	\mathcal{D}_{c} (z)_{\xr} =
	\frac12\left( \mathcal{D}_{+}(z) + \mathcal{D}_{-}(z) \right)_{\xr} =  \frac{1}{6\Delta x}\left( z_{\xl} - 8z_i + 8z_{i+1} - z_{i+\frac32} \right).
\end{equation}
The $4$th-order compact operator of the second derivative $\mathcal{D}^{2}$ is adapted from \cite{Kamakoti_2010_High_SJSC},
\begin{equation}\label{eq:1d_diff_compact_kappa_u}
	\mathcal{D}^{2} (\bK, \bZ)_{\xr} = \dfrac{4}{\Delta x^2}\sum_{m=-2}^{2}\sum_{n=-2}^{2} b_{mn}\bK\left(\bW_{i+\frac{m+1}{2}}\right) \bZ_{i+\frac{n+1}{2}}.
\end{equation}
Here the coefficients $b_{mn}$ are obtained from the order conditions, a sparsity requirement, and symmetry.
More precisely, let
\begin{equation}\label{eq:ns_compact_B}
	B = \begin{bmatrix}
		0 & 0 \\ 0 & M \\
	\end{bmatrix}
	- \begin{bmatrix}
		M & 0 \\ 0 & 0 \\
	\end{bmatrix}, \quad
	M =
	\begin{bmatrix}
		\frac{1}{8} & -\frac{1}{6} & \frac{1}{24} & 0 \\[3pt]
		-\frac{1}{6} & -\frac{3}{8} & \frac{2}{3} & -\frac{1}{8} \\[3pt]
		\frac{1}{8} & -\frac{2}{3} & \frac{3}{8} & \frac{1}{6} \\[3pt]
		0 & -\frac{1}{24} & \frac{1}{6} & -\frac{1}{8} \\
	\end{bmatrix},
\end{equation}
and the rows and columns of $B$ are ordered by $\{-2,-1,0,1,2\}$, then $b_{mn}$ denotes the entry associated with row index $m$ and column index $n$ in this ordering.
Note that the matrix $\bK(\bW_{i+\frac{m+1}{2}})$ is evaluated separately at every stencil point,
and the right-hand side in \eqref{eq:af_pnt_1d} discretizes $\partial_x(\bK\partial_x\bZ)$ directly.

\subsection{The 2D formulation}\label{sec:2d_formulation}
The 2D compressible NS equations can be written as
\begin{equation}\label{eq:ns_2d}
	\partial_t\bU + \sum_{\ell=1}^{2}\partial_{x_\ell} \bF_\ell^{c}
	= \sum_{\ell=1}^{2}\partial_{x_\ell} \bF_\ell^{v}
	= \sum_{\ell,m=1}^{2} \partial_{x_\ell}(\bK_{\ell m}(\bW) \partial_{x_m}\bZ),
\end{equation}
with $\bU=(\rho,\rho u,\rho v,E)^{\top}$, $\bW=(\rho,u,v,p)^{\top}$, $\bZ=(0,u,v,T)^{\top}$, and $(x_1,x_2)=(x,y)$.
The viscous fluxes can be expressed as
\begin{equation*}%\label{eq:ns_2d_viscous_decomposition}
	\bF^v_1=\bK_{11}(\bW)\partial_x\bZ+\bK_{12}(\bW)\partial_y\bZ,
	\quad
	\bF^v_2=\bK_{21}(\bW)\partial_x\bZ+\bK_{22}(\bW)\partial_y\bZ.
\end{equation*}
The four $4\times4$ matrices are
\begin{align*}
	\bK_{11}(\bW)&=
	\begin{pmatrix}
		0&0&0&0\\
		0&\frac43\mu&0&0\\
		0&0&\mu&0\\
		0&\frac43\mu u&\mu v&\kappa
	\end{pmatrix},
	&\bK_{22}(\bW)&=
	\begin{pmatrix}
		0&0&0&0\\
		0&\mu&0&0\\
		0&0&\frac43\mu&0\\
		0&\mu u&\frac43\mu v&\kappa
	\end{pmatrix},\\
	\bK_{12}(\bW)&=
	\begin{pmatrix}
		0&0&0&0\\
		0&0&-\frac23\mu&0\\
		0&\mu&0&0\\
		0&\mu v&-\frac23\mu u&0
	\end{pmatrix},
	&\bK_{21}(\bW)&=
	\begin{pmatrix}
		0&0&0&0\\
		0&0&\mu&0\\
		0&-\frac23\mu&0&0\\
		0&-\frac23\mu v&\mu u&0
	\end{pmatrix}.
\end{align*}

Consider a uniform Cartesian mesh with $N_1\times N_2$ cells
$I_{ij} = [x_{\xl}, x_{\xr}]\times [y_{\yl}, y_{\yr}]$ with cell centers $((x_{\xl} + x_{\xr})/2, (y_{\yl} + y_{\yr})/2$ and cell sizes $\Delta x, \Delta y$.
The DoFs consist of the cell averages, edge-centered, and corner point values of the numerical solution $\bU_h(\bx, t)$, defined as
\begin{equation*}
	\overline{\bU}_{\bm{i}}(t) = \frac{1}{\Delta x\Delta y}\int_{I_{\bm{i}}} \bU_h(\bx,t) ~\dd \bx,\quad
	\bU_{\bm{\zeta}}(t) = \bU_h(\bx_{\bm{\zeta}}, t).
\end{equation*}
Here $\bm{i}=(i,j)$, $\bx=(x,y)=(x_1,x_2)$, and $\bm{\zeta} \in \{ (i+\frac12,j), (i,j+\frac12), (i+\frac12,j+\frac12) \}$ denotes the locations of the point values.
The proposed AF method is as follows,
\begin{align}
	&\frac{\dd \overline{\bU}_{\bm{i}}}{\dd t} + \sum_{\ell=1}^2\frac{1}{\Delta x_{\ell}}\left( \widehat{\bF}^{c}_{\ell,\bm{i}+\frac12\bm{e}_\ell} - \widehat{\bF}^{c}_{\ell,\bm{i}-\frac12\bm{e}_\ell} \right) =
	\sum_{\ell=1}^2\frac{1}{\Delta x_{\ell}}\left( \widehat{\bF}^{v}_{\ell,\bm{i}+\frac12\bm{e}_\ell} - \widehat{\bF}^{v}_{\ell,\bm{i}-\frac12\bm{e}_\ell} \right), \nonumber\\
	&\widehat{\bF}^{c}_{\ell,\bm{i}+\frac12\bm{e}_\ell} = \overline{ \bF^{c}_{\ell}(\bU) }_{\bm{i}+\frac12\bm{e}_\ell},
	\quad
	\widehat{\bF}^{v}_{\ell,\bm{i}+\frac12\bm{e}_\ell} = \overline{ \bF^{v}_{\ell}(\bU,\nabla\bZ) }_{\bm{i}+\frac12\bm{e}_\ell}, \nonumber\\
	&\left(\pd{\bZ}{x_\ell}\right)_{\bm{i}+\frac12\bm{e}_\ell} = \mathcal{D}^{(\ell)}_{c} (\bZ)_{\bm{i}+\frac12\bm{e}_\ell}, \nonumber\\
	&\frac{\dd \bU_{\bm{\zeta}} }{\dd t} + \sum_{\ell=1}^2 \left(\mathcal{D}^{(\ell)}_{+} (\bF^{c,+}_{\ell})_{\bm{\zeta}} + \mathcal{D}^{(\ell)}_{-}(\bF^{c,-}_{\ell})_{\bm{\zeta}} \right) = \sum_{\ell=1}^2 \mathcal{D}^{(\ell),2} (\bK_{\ell\ell}(\bW), \bZ)_{\bm{\zeta}} \nonumber \\
	&+ \sum_{\ell,m=1,\ell\not=m}^2 \mathcal{D}_{c}^{(\ell)} (\bK_{\ell m}(\bW) \mathcal{D}_{c}^{(m)} \bZ)_{\bm{\zeta}}, \label{eq:2d_semi_af}
\end{align}
where the overbar on each flux denotes its average along the corresponding edge, approximated by Simpson's rule,
and the mixed derivatives in the point value update are discretized by using the operators in each direction sequentially.
Here, the finite difference operators $\mathcal{D}^{(\ell)}_{\pm}$, $\mathcal{D}^{(\ell)}_{c}$, and $\mathcal{D}^{(\ell),2}$ are applied along the $x_{\ell}$-direction, based on \eqref{eq:1d_upwind_ops_4p}-\eqref{eq:1d_upwind_ops_4m}, \eqref{eq:1d_diff_central_4th}, and \eqref{eq:1d_diff_compact_kappa_u}, respectively.
For example,
\begin{align*}%\label{eq:1d_upwind_ops_4p}
	(\mathcal{D}_{+}^{(1)}z)_{i+\frac12,j} &= \frac{1}{3\Delta x}\left( z_{i-\frac12,j} - 6z_{i,j} + 3z_{i+\frac12,j} + 2z_{i+1,j} \right), \\
	(\mathcal{D}_{-}^{(1)}z)_{i+\frac12,j} &= \frac{1}{3\Delta x}\left( -2z_{i,j} - 3z_{i+\frac12,j} + 6z_{i+1,j} - z_{i+\frac32,j} \right), \\
	(\mathcal{D}_{+}^{(2)}z)_{i,j+\frac12} &= \frac{1}{3\Delta y}\left( z_{i,j-\frac12} - 6z_{i,j} + 3z_{i,j+\frac12} + 2z_{i,j+1} \right), \\
	(\mathcal{D}_{-}^{(2)}z)_{i,j+\frac12} &= \frac{1}{3\Delta y}\left( -2z_{i,j} - 3z_{i,j+\frac12} + 6z_{i,j+1} - z_{i,j+\frac32} \right).
\end{align*}
For each conservative component $z$, the cell-centered point value is obtained through
\begin{align*}
	z_{i,j} = \frac{1}{16}\Big[
	36\overline{z}_{i,j} &- 4\left(z_{\xl,j}+z_{\xr,j}+z_{i,\yl}+z_{i,\yr}\right)\nonumber \\
	&- \left(z_{\xl,\yl} + z_{\xr,\yl} + z_{\xl,\yr} + z_{\xr,\yr}\right)
	\Big].
\end{align*}
The central finite difference operator is
\begin{equation*}
	(\mathcal{D}_{c}^{(\ell)} z)_{i,j} = \frac12\left( (\mathcal{D}_{+}^{(\ell)}z)_{i,j} + (\mathcal{D}_{-}^{(\ell)}z)_{i,j} \right).
\end{equation*}
The compact discretization for $\partial_{x}(\bK_{11}(\bW) \partial_{x}\bZ),$ can be expressed as
\begin{equation*}
	\mathcal{D}^{(1),2} (\bK, \bZ)_{\xr,j} = \dfrac{4}{\Delta x^2}\sum_{m=-2}^{2}\sum_{n=-2}^{2} b_{mn}\bK\left(\bW_{i+\frac{m+1}{2},j}\right) \bZ_{i+\frac{n+1}{2},j}.
\end{equation*}
The remaining formulas follow by half-index shifts in $i$ and/or $j$.

\subsection{Time discretization}
For the semi-discretizations \eqref{eq:1d_semi_af} and \eqref{eq:2d_semi_af},
this paper uses the explicit five-stage, fourth-order strong-stability-preserving (SSP) RK method SSP-RK54 for time integration.
The CFL condition is
\begin{equation}\label{eq:cfl}
	\Delta t\leqslant
	\dfrac{ 1 }
	{
		\max\limits_{\bm{i},\ell}\dfrac{ \varrho_{\ell}( \overline{\bU}_{\bm{i}}) }{ C^{c}_{\texttt{CFL}}\Delta x_\ell }
		+ \max\limits_{\bm{i},\ell}\dfrac{ \widetilde{\varrho}( \overline{\bU}_{\bm{i}}) }{ C^{v}_{\texttt{CFL}}\Delta x_\ell^2} },
\end{equation}
where $\varrho_{\ell} = \abs{u_{\ell}} + c$ and $\widetilde{\varrho} = \max\{ \frac{4\mu}{3\rho}, \frac{\gamma\mu}{\Pr\rho} \}$ are the maximum absolute values of the eigenvalues of the inviscid and viscous parts, respectively.
The CFL numbers $C^{c}_{\texttt{CFL}} = 1.04/d$ and $C^{v}_{\texttt{CFL}}= 0.33/d$ are chosen based on the von Neumann analysis in \cite{Duan_2026_compact} for the inviscid and viscous parts, respectively.
Note that, for the 1D linear advection equation, the stable CFL limit is higher than that of the standard $3$rd-order AF method \cite{Abgrall_2023_Extensions_EMMNA},
which is another advantage of using the additional downwind point in the biased stencil.

\section{Positivity-preserving limiting and oscillation control}\label{sec:ns_pp}

The high-order AF method may produce negative density or pressure,
leading to unphysical solutions or even causing the simulations to blow up.
This section first presents a monolithic convex limiting approach that constrains individual numerical fluxes to preserve density and pressure positivity in the cell average update, adapting the approach proposed in \cite{Kuzmin_2020_Monolithic_CMAME} to the NS equations.

\begin{definition}
	An AF method is called positivity-preserving (PP) if starting from cell averages and point values in the admissible state set
	\begin{equation}\label{eq:ns_g}
		\mathcal{G} = \left\{\bU = \left(\rho, \rho\bu, E\right) ~\Big|~ \rho > 0,~ p = (\gamma-1)\left(E - \frac{\norm{\rho\bu}^2}{2\rho}\right) > 0 \right\},
	\end{equation}
	the cell averages and point values remain in $\mathcal G$ at the next time step.
\end{definition}

The PP properties of the limiting are shown using representations in terms of intermediate states that stay in the convex admissible state set \eqref{eq:ns_g}.
Positivity is proved separately for the inviscid and viscous split states, which are then combined in the final limiting procedure.

\subsection{Preliminaries}

The following statement for the inviscid part is the usual one for the Euler equations.

\begin{lemma}[LF splitting for the inviscid part]\label{lem:ns_convective_split}
	Let $\bU\in\mathcal G$ and the sound speed $c=\sqrt{\gamma p/\rho}$,
	the following statement holds,
	\begin{equation*}
		\bU\pm\frac{\bF_{\ell}^c(\bU)}{\alpha_{\ell}}\in\mathcal G, \quad \text{if}~\alpha_{\ell} \geqslant \varrho_{\ell} = |u_\ell|+c.
	\end{equation*}
\end{lemma}

The following useful PP property of the directional LF splitting can be found in \cite{Zhang_2017_positivity_JCP}.

\begin{lemma}[LF splitting for the viscous part]\label{lem:ns_viscous_split}
	Let $\bU\in\mathcal G$, $\bm{\tau}_\ell=\btau\bm e_\ell$, $q_\ell=\bq\cdot\bm{e}_\ell$,
	and introduce an auxiliary variable $\bS$ as an approximation to $\nabla\bZ$.
	Define
	\begin{equation*}
		b_\ell(\bU,\bS)=
		\frac{\sqrt{\rho^2q_{\ell}^2+2\rho^2e\norm{\bm{\tau}_\ell}^2}
			+\rho\abs{q_{\ell}}}
		{2\rho^2e},
	\end{equation*}
	then
	\begin{equation*}
		\bU\pm\frac{\bF_{\ell}^v(\bU,\bS)}{\beta_\ell}
		\in\mathcal G, \quad
		\text{if}~ \beta_\ell > b_\ell.
	\end{equation*}
\end{lemma}

\begin{proof}
	The mass component of $\bF_{\ell}^v$ is zero, so density positivity follows from $\bU\in\mathcal{G}$.
	Using $\bF_{\ell}^v=(0,\bm \tau_{\ell},\bu\cdot\bm \tau_{\ell}-q_{\ell})^{\top}$ gives
	\begin{equation*}
		\rho e\left(\bU\pm\frac{\bF_{\ell}^v}{\beta_{\ell}}\right)
		=\rho e(\bU)\mp\frac{q_{\ell}}{\beta_{\ell}}
		-\frac{\norm{\bm \tau_{\ell}}^2}{2\rho\beta_{\ell}^2}.
	\end{equation*}
	It is positive if $\rho e > \frac{\abs{q_{\ell}}}{\beta_{\ell}}
	+ \frac{\norm{\bm \tau_{\ell}}^2}{2\rho\beta_{\ell}^2}$,
	which yields the condition $\beta_\ell > b_\ell$.
\end{proof}

\begin{lemma}[LF splitting for the total flux]\label{lem:ns_total_flux_split}
	Let
	\begin{equation}\label{eq:ns_total_physical_flux_speed}
		\bF_\ell(\bU,\bS)=\bF_{\ell}^c(\bU)-\bF_{\ell}^v(\bU,\bS),
		\quad \lambda_\ell=\alpha_{\ell}+\beta_\ell,
	\end{equation}
	where $\alpha_{\ell}$ and $\beta_\ell$ satisfy the conditions in
	Lemmas~\ref{lem:ns_convective_split} and~\ref{lem:ns_viscous_split}, respectively.
	Then
	\begin{equation}\label{eq:ns_total_flux_split_states}
		\bU\pm\frac{\bF_\ell(\bU,\bS)}{\lambda_\ell}\in\mathcal G.
	\end{equation}
\end{lemma}

\begin{proof}
	The identity
	\begin{equation*}
		\bU\pm\frac{\bF_\ell}{\lambda_\ell}
		=\frac{\alpha_{\ell}}{\lambda_\ell}
		\left(\bU\pm\frac{\bF_{\ell}^c}{\alpha_{\ell}}\right)
		+\frac{\beta_\ell}{\lambda_\ell}
		\left(\bU\mp\frac{\bF_{\ell}^v}{\beta_\ell}\right)
	\end{equation*}
	is a convex combination,
	so the proof follows from the convexity of $\mathcal G$.
\end{proof}

For $d\in\{1,2\}$, consider the following low-order LLF scheme for the cell average,
\begin{align}
	&\overline{\bU}_{\bm{i}}^{\texttt{L}} = \overline{\bU}_{\bm{i}}^{n} - \sum_{\ell=1}^{d}\mu_{\ell} \left(\widehat{\bF}^{\texttt{L}}_{\bm{i} + \frac12\bm{e}_\ell} - \widehat{\bF}^{\texttt{L}}_{\bm{i} - \frac12\bm{e}_\ell}\right), \label{eq:llf_xl} \\
	&\widehat{\bF}^{\texttt{L}}_{\bm{i} + \frac12\bm{e}_\ell}
	=\widehat{\bF}^{c,\texttt{L}}_{\bm{i} + \frac12\bm{e}_\ell} - \widehat{\bF}^{v,\texttt{L}}_{\bm{i} + \frac12\bm{e}_\ell},
	\quad
	\mu_{\ell} = \Delta t^n / \Delta x_{\ell}, \nonumber
\end{align}
where
\begin{subequations}\label{eq:ns_low_face_fluxes}
	\begin{align}
		&\widehat{\bF}^{c,\texttt{L}}_{\bm{i}+\frac12\bm{e}_{\ell}} = \frac12\left(\bF_{\ell}^{c}(\overline{\bU}_{\bm{i}}) + \bF_{\ell}^{c}(\overline{\bU}_{\bm{i} + \bm{e}_{\ell}})\right)
		- \frac{\alpha_{\bm{i} + \frac12\bm{e}_{\ell}}}{2}\left(\overline{\bU}_{\bm{i}+\bm{e}_{\ell}} - \overline{\bU}_{\bm{i}}\right), \label{eq:ns_low_convective_flux}\\
		&\widehat{\bF}^{v,\texttt{L}}_{\bm{i}+\frac12\bm{e}_{\ell}} = \frac12\left(\bF_{\ell}^{v}(\overline{\bU}_{\bm{i}}, \bS_{\bm{i}}) + \bF_{\ell}^{v}(\overline{\bU}_{\bm{i} + \bm{e}_{\ell}}, \bS_{\bm{i} + \bm{e}_{\ell}})\right)
		+ \frac{\beta_{\bm{i} + \frac12\bm{e}_{\ell}}}{2}\left(\overline{\bU}_{\bm{i}+\bm{e}_{\ell}} - \overline{\bU}_{\bm{i}}\right). \label{eq:ns_low_viscous_flux}
	\end{align}
\end{subequations}
Here the low-order gradient $\bS_{\bm{i}}$ is computed using the central difference of $\bZ(\overline{\bU}_{\bm{i}})$.
Note that the superscript $n$ is omitted for brevity.
Define the total flux
\begin{equation*}
	\bF_{\ell,\bm{i}} = \bF_{\ell}^c(\overline{\bU}_{\bm{i}}) - \bF_{\ell}^v(\overline{\bU}_{\bm{i}}, \bS_{\bm{i}}),
\end{equation*}
then the total numerical flux can also be expressed as
\begin{equation}\label{eq:ns_low_total_llf_identity}
	\widehat{\bF}^{\texttt{L}}_{\bm{i} + \frac12\bm{e}_\ell}
	= \frac12\left(\bF_{\ell,\bm{i}} + \bF_{\ell,\bm{i} + \bm{e}_{\ell}} \right)
	- \frac{\lambda_{\bm{i}+ \frac12\bm{e}_\ell}}{2}
	\left(\overline{\bU}_{\bm{i} + \bm{e}_{\ell}}-\overline{\bU}_{\bm{i}}\right),
\end{equation}
where
\begin{equation*}
	\lambda_{\bm{i}+\frac12\bm{e}_{\ell}}=\alpha_{\bm{i}+\frac12\bm{e}_{\ell}}+\beta_{\bm{i}+\frac12\bm{e}_{\ell}}.
\end{equation*}
The numerical viscosity coefficients are chosen as
\begin{align*}
	\alpha_{\bm{i}+ \frac12\bm{e}_\ell}
	&= \max\left\{ \varrho_{\ell}(\overline{\bU}_{\bm{i}}),
	\varrho_{\ell}(\overline{\bU}_{\bm{i} + \bm{e}_{\ell}}) \right\},\\
	\beta_{\bm{i}+ \frac12\bm{e}_\ell}
	&= \max\left\{ b_\ell(\overline{\bU}_{\bm{i}}, \bS_{\bm{i}}),
	b_\ell(\overline{\bU}_{\bm{i} + \bm{e}_{\ell}}, \bS_{\bm{i} + \bm{e}_{\ell}}) \right\}.
\end{align*}
With the above notations, the LLF scheme \eqref{eq:llf_xl} can be written as a convex decomposition
\begin{equation*}
	\begin{aligned}
		\overline{\bU}_{\bm{i}}^{\texttt{L}} = \frac{1}{2d}\sum_{\ell=1}^{d}\left( \bU_{\ell,\bm{i}}^{\texttt{L},+} + \bU_{\ell,\bm{i}}^{\texttt{L},-} \right),
	\end{aligned}
\end{equation*}
where
\begin{equation}\label{eq:ns_1d_monolithic_intermediate_states}
	\begin{aligned}
		&\bU_{\ell,\bm{i}}^{\texttt{L},+} = \overline{\bU}_{\bm{i}} - 2d\mu_{\ell}\left(\widehat{\bF}_{\bm{i} + \frac12\bm{e}_{\ell}}^{\texttt{L}}-\bF_{\ell,\bm{i}}\right), \\
		&\bU_{\ell,\bm{i}}^{\texttt{L},-} = \overline{\bU}_{\bm{i}} - 2d\mu_{\ell}\left(\bF_{\ell,\bm{i}}-\widehat{\bF}_{\bm{i} - \frac12\bm{e}_{\ell}}^{\texttt{L}}\right).
	\end{aligned}
\end{equation}
To show the PP property of the LLF scheme, it suffices to show that the intermediate states stay in $\mathcal{G}$ due to the convexity of $\mathcal{G}$.

\begin{proposition}\label{thm:ns_1d_intermediate_positive}
	Assume that $\overline{\bU}_{\bm{i}}\in\mathcal G$ for all $\bm{i}$,
	and the numerical viscosity coefficients satisfy the conditions in
	Lemmas~\ref{lem:ns_convective_split} and~\ref{lem:ns_viscous_split}.
	Then the intermediate states in \eqref{eq:ns_1d_monolithic_intermediate_states} satisfy $\bU_{\ell,\bm{i}}^{\texttt{L},\pm} \in \mathcal G$, under the following time-step size constraint
	\begin{equation}\label{eq:ns_1d_monolithic_cfl}
		\begin{aligned}
			&\Delta t\leqslant\min_{1\leqslant\ell\leqslant d}\frac{\Delta x_\ell}{2d\max\limits_{\bm{i}}\{\lambda_{\bm i+\frac12\bm e_\ell}\}}.
		\end{aligned}
	\end{equation}
\end{proposition}

\begin{proof}
	For the two interfaces, let
	$\nu_{\ell,\pm}=\mu_\ell\lambda_{\bm i\pm\frac12\bm e_\ell}$, respectively.
	The intermediate states can be expressed as convex decompositions
	\begin{align*}
		&\bU_{\ell,\bm{i}}^{\texttt{L},+}=\left(1-2d\nu_{\ell,+}\right)\overline{\bU}_{\bm{i}}+d\nu_{\ell,+}\left(\overline{\bU}_{\bm{i}}+\frac{\bF_{\ell,\bm{i}}}{\lambda_{\bm{i}+\frac12\bm{e}_{\ell}}}\right)+d\nu_{\ell,+}\left(\overline{\bU}_{\bm{i}+\bm{e}_{\ell}}-\frac{\bF_{\ell,\bm{i}+\bm{e}_{\ell}}}{\lambda_{\bm{i}+\frac12\bm{e}_{\ell}}}\right), \\
		&\bU_{\ell,\bm{i}}^{\texttt{L},-}=\left(1-2d\nu_{\ell,-}\right)\overline{\bU}_{\bm{i}}+d\nu_{\ell,-}\left(\overline{\bU}_{\bm{i}}-\frac{\bF_{\ell,\bm{i}}}{\lambda_{\bm{i}-\frac12\bm{e}_{\ell}}}\right)+d\nu_{\ell,-}\left(\overline{\bU}_{\bm{i}-\bm{e}_{\ell}}+\frac{\bF_{\ell,\bm{i}-\bm{e}_{\ell}}}{\lambda_{\bm{i}-\frac12\bm{e}_{\ell}}}\right).
	\end{align*}
	Under \eqref{eq:ns_1d_monolithic_cfl}, the coefficients $1-2d\nu_{\ell,\pm}$, $d\nu_{\ell,\pm}$, and $d\nu_{\ell,\pm}$ are nonnegative and sum to one.
	Using Lemma~\ref{lem:ns_total_flux_split} and the convexity of $\mathcal{G}$, one has $\bU_{\ell,\bm{i}}^{\texttt{L},\pm} \in \mathcal G$.
\end{proof}

It is known that the low-order scheme is PP and non-oscillatory,
thus the basic idea of the limiting for cell average is to blend the high-order AF numerical flux with the low-order LLF flux, which guarantees local conservation.
To be specific, at the cell interface $\bm i+\frac12\bm e_\ell$, define the total flux in the high-order AF method as
\begin{equation*}
	\widehat{\bF}^{\texttt{H}}_{\bm i+\frac12\bm e_\ell}
	=\widehat{\bF}^{c,\texttt{H}}_{\bm i+\frac12\bm e_\ell}
	-\widehat{\bF}^{v,\texttt{H}}_{\bm i+\frac12\bm e_\ell},
\end{equation*}
and denote the anti-diffusion term by
\begin{equation*}
	\Delta \widehat{\bF}_{\bm i+\frac12\bm e_\ell} = \widehat{\bF}^{\texttt H}_{\bm i+\frac12\bm e_\ell}
	-\widehat{\bF}^{\texttt L}_{\bm i+\frac12\bm e_\ell}.
\end{equation*}
Then the modified limited flux can be generally expressed as
\begin{equation*}
	\widehat{\bF}^{\texttt{modified}}_{\bm i+\frac12\bm e_\ell}
	=\widehat{\bF}^{\texttt L}_{\bm i+\frac12\bm e_\ell}
	+ \theta^{\texttt{modified}}_{\bm i+\frac12\bm e_\ell}
	\Delta \widehat{\bF}_{\bm i+\frac12\bm e_\ell},
\end{equation*}
where $\theta^{\texttt{modified}}_{\bm i+\frac12\bm e_\ell} \in [0,1]$ is determined by a smoothness indicator or by density and pressure positivity requirements.
The blending parameters are determined sequentially below.

\subsection{Smoothness indicator based limiting}
This paper adopts the smoothness indicator in \cite{Duan_2025_Active_SJSC},
which combines Jameson's shock sensor in \cite{Jameson_1981_Solutions_AJ} and a modified Ducros shock sensor \cite{Ducros_1999_Large_JCP}.
For $d\in\{1,2\}$, let
\begin{equation*}
	(\varphi_1)_{\bm{i}}
	=\max_{1\leqslant m\leqslant d}
	\dfrac{\abs{\bar{p}_{\bm{i}+\bm{e}_{m}}
			-2\bar{p}_{\bm{i}}+\bar{p}_{\bm{i}-\bm{e}_{m}}}}
	{\abs{\bar{p}_{\bm{i}+\bm{e}_{m}}
			+2\bar{p}_{\bm{i}}+\bar{p}_{\bm{i}-\bm{e}_{m}}}},
\end{equation*}
and
\begin{equation*}
	(\varphi_2)_{\bm{i}}
	=\max\left\{
	\dfrac{-(\nabla\cdot\bar{\bu})_{\bm{i}}}
	{\sqrt{(\nabla\cdot\bar{\bu})_{\bm{i}}^2
			+\norm{(\nabla\times\bar{\bu})_{\bm{i}}}^2+10^{-13}}},~0
	\right\}.
\end{equation*}
where $(\nabla\cdot\bar{\bu})_{\bm{i}}$ and $(\nabla\times\bar{\bu})_{\bm{i}}$ are obtained by using central differences,
with $\bar{\bu}_{\bm{i}}$ and $\bar{p}_{\bm{i}}$ the velocity and pressure recovered from the cell average $\overline{\bU}_{\bm{i}}$.
%We consider the sign of the velocity divergence, such that the shock waves can be located better.
The blending coefficient is designed as
\begin{align*}
	&\theta_{\bm{i}+\frac12\bm{e}_{\ell}}^{\texttt{s}} = \exp(-C_{\texttt{s}} (\varphi_1)_{\bm{i}+\frac12\bm{e}_{\ell}} (\varphi_2)_{\bm{i}+\frac12\bm{e}_{\ell}})\in [0, 1],\\
	&(\varphi_s)_{\bm{i}+\frac12\bm{e}_{\ell}} = \max\left\{(\varphi_s)_{\bm{i}}, (\varphi_s)_{\bm{i}+\bm{e}_{\ell}}\right\}, ~s=1,2,
\end{align*}
where the parameter $C_{\texttt{s}}$ adjusts the strength of the limiting, which is fixed as $1$ throughout this paper.
The limited numerical flux after this limiting is
\begin{equation}\label{eq:limited_s}
	\widehat{\bF}_{\bm{i} + \frac12\bm{e}_{\ell}}^{\texttt{Lim},\texttt{s}} 
	= \widehat{\bF}_{\bm{i} + \frac12\bm{e}_{\ell}}^{\texttt{L}}
	+ \theta_{\bm{i} + \frac12\bm{e}_{\ell}}^{\texttt{s}}\Delta\widehat{\bF}_{\bm{i}
		+ \frac12\bm{e}_{\ell}},
\end{equation}
and the corresponding anti-diffusion term is updated as
\begin{equation*}
	\Delta\widehat{\bF}_{\bm{i} + \frac12\bm{e}_{\ell}}^{\texttt{Lim},\texttt{s}}
	= \widehat{\bF}_{\bm{i} + \frac12\bm{e}_{\ell}}^{\texttt{Lim},\texttt{s}}
	- \widehat{\bF}_{\bm{i} + \frac12\bm{e}_{\ell}}^{\texttt{L}}.
\end{equation*}

\subsection{Positivity-preserving limiting}
Similar to the low-order scheme \eqref{eq:llf_xl}, the full cell average update in the AF method can be decomposed into $2d$ intermediate states as follows,
\begin{equation*}
	\overline{\bU}_{\bm{i}}^{\texttt{H}} = \frac{1}{2d}\sum_{\ell=1}^{d}\left( \bU_{\ell,\bm{i}}^{\texttt{H},+} + \bU_{\ell,\bm{i}}^{\texttt{H},-} \right),
\end{equation*}
where
\begin{equation*}
	\begin{aligned}
		&\bU_{\ell,\bm{i}}^{\texttt{H},+} = \overline{\bU}_{\bm{i}} - 2d\mu_{\ell}\left(\widehat{\bF}_{\bm{i} + \frac12\bm{e}_{\ell}}^{\texttt{Lim},\texttt{s}}-\bF_{\ell,\bm{i}}\right), \\
		&\bU_{\ell,\bm{i}}^{\texttt{H},-} = \overline{\bU}_{\bm{i}} - 2d\mu_{\ell}\left(\bF_{\ell,\bm{i}}-\widehat{\bF}_{\bm{i} - \frac12\bm{e}_{\ell}}^{\texttt{Lim},\texttt{s}}\right),
	\end{aligned}
\end{equation*}
with $\widehat{\bF}_{\bm{i} \pm \frac12\bm{e}_{\ell}}^{\texttt{Lim},\texttt{s}}$ given by \eqref{eq:limited_s}.
The superscript $\texttt H$ denotes the high-order candidate after the limiting based on the smoothness indicator.
Define the limited cell average update as
\begin{equation}\label{eq:limited_xl}
	\overline{\bU}_{\bm{i}}^{\texttt{Lim},\texttt{pp}} = \overline{\bU}_{\bm{i}}^{n} - \sum_{\ell=1}^{d}\mu_{\ell} \left(\widehat{\bF}^{\texttt{Lim},\texttt{pp}}_{\bm{i} + \frac12\bm{e}_\ell} - \widehat{\bF}^{\texttt{Lim},\texttt{pp}}_{\bm{i} - \frac12\bm{e}_\ell}\right),
\end{equation}
where the limited numerical flux is given by
\begin{equation*}
	\widehat{\bF}^{\texttt{Lim},\texttt{pp}}_{\bm{i} + \frac12\bm{e}_\ell}
	= \widehat{\bF}^{\texttt{L}}_{\bm{i} + \frac12\bm{e}_\ell}
	+ \theta^{\texttt{pp}}_{\bm{i} + \frac12\bm{e}_\ell} \Delta\widehat{\bF}^{\texttt{Lim},\texttt{s}}_{\bm{i} + \frac12\bm{e}_\ell}.
\end{equation*}
Rewrite the limited cell average as
\begin{equation*}%\label{eq:ns_limited_cell_convex_decomposition}
	\overline{\bU}_{\bm{i}}^{\texttt{Lim},\texttt{pp}} = \frac{1}{2d}\sum_{\ell=1}^{d}\left( \bU_{\ell,\bm{i}}^{\texttt{Lim},\texttt{pp},+} + \bU_{\ell,\bm{i}}^{\texttt{Lim},\texttt{pp},-} \right),
\end{equation*}
with the limited intermediate states
\begin{align*}
	&\bU_{\ell,\bm{i}}^{\texttt{Lim},\texttt{pp},+}=\bU_{\ell,\bm{i}}^{\texttt{L},+}-2d\mu_{\ell}\theta^{\texttt{pp}}_{\bm{i} + \frac12\bm{e}_\ell}\Delta\widehat{\bF}^{\texttt{Lim},\texttt{s}}_{\bm{i} + \frac12\bm{e}_\ell} =(1-\theta^{\texttt{pp}}_{\bm{i} + \frac12\bm{e}_\ell})\bU_{\ell,\bm{i}}^{\texttt{L},+}+\theta^{\texttt{pp}}_{\bm{i} + \frac12\bm{e}_\ell}\bU_{\ell,\bm{i}}^{\texttt{H},+}, \\
	&\bU_{\ell,\bm{i}}^{\texttt{Lim},\texttt{pp},-}=\bU_{\ell,\bm{i}}^{\texttt{L},-}+2d\mu_{\ell}\theta^{\texttt{pp}}_{\bm{i} - \frac12\bm{e}_\ell}\Delta\widehat{\bF}^{\texttt{Lim},\texttt{s}}_{\bm{i} - \frac12\bm{e}_\ell} =(1-\theta^{\texttt{pp}}_{\bm{i} - \frac12\bm{e}_\ell})\bU_{\ell,\bm{i}}^{\texttt{L},-}+\theta^{\texttt{pp}}_{\bm{i} - \frac12\bm{e}_\ell}\bU_{\ell,\bm{i}}^{\texttt{H},-},
\end{align*}
where the blending parameter $\theta^{\texttt{pp}}_{\bm{i} \pm \frac12\bm{e}_\ell}\in[0,1]$.
With the above construction, one can verify the following proposition.

\begin{proposition}\rm
	If the limited intermediate states $\bU_{\ell,\bm{i}}^{\texttt{Lim},\texttt{pp},\pm} \in \mathcal{G}$ for all $\bm{i}$ and $\ell=1,\ldots,d$,
	then the limited average update \eqref{eq:limited_xl} is PP,
	i.e., $\overline{\bU}^{\texttt{Lim},\texttt{pp}}_{\bm{i}} \in \mathcal{G}$, under the CFL condition \eqref{eq:ns_1d_monolithic_cfl}.
	If the limiting is applied to every forward-Euler stage, the SSP-RK method preserves the PP property through its convex stage combinations.
\end{proposition}

The remaining task is to determine the parameter $\theta^{\texttt{pp}}_{\bm{i}\pm\frac12\bm{e}_{\ell}} \in [0,1]$ at each interface,
such that $\bU_{\ell,\bm{i}}^{\texttt{Lim},\texttt{pp},\pm} \in \mathcal{G}$ and $\theta^{\texttt{pp}}_{\bm{i}\pm\frac12\bm{e}_{\ell}}$ is as large as permitted.
Take the computation of $\theta_{\bm{i}+\frac12\bm{e}_{\ell}}$ as an example.
Define
\begin{equation*}
	\Theta
	\left(
	z^{\texttt{L}}, z^{\texttt{H}}, \varepsilon
	\right)
	=
	\begin{cases}
		\dfrac{z^{\texttt{L}} - \varepsilon}
		{z^{\texttt{L}}-z^{\texttt{H}}},
		&
		z^{\texttt{H}} < \varepsilon,
		\\
		1,
		&
		\text{otherwise},
	\end{cases}
\end{equation*}
for $z^{\texttt{L}} \geqslant \varepsilon$.
By construction, it holds $(1-\vartheta)z^{\texttt{L}} + \vartheta z^{\texttt{H}} \geqslant \varepsilon$ for all $\vartheta \in [0, \Theta(z^{\texttt{L}},z^{\texttt{H}}, \varepsilon)]$.

The first step is to impose the density positivity
\begin{equation*}
	\rho(\widetilde{\bU}_{\ell,\bm{i}}^{\texttt{Lim},\rho,+}) > \varepsilon^\rho_{\ell,\bm{i}}, \quad
	\rho(\widetilde{\bU}_{\ell,\bm{i}+\bm{e}_{\ell}}^{\texttt{Lim},\rho,-}) > \varepsilon^\rho_{\ell,\bm{i}+\bm{e}_{\ell}},
\end{equation*}
such that the following intermediate states have positive density component,
\begin{equation*}
	\widetilde{\bU}_{\ell,\bm{i}}^{\texttt{Lim},\rho,\pm}
	= (1 - \theta^\rho_{\bm{i} \pm \frac12\bm{e}_\ell})\bU_{\ell,\bm{i}}^{\texttt{L},\pm}
	+ \theta^\rho_{\bm{i} \pm \frac12\bm{e}_\ell} \bU_{\ell,\bm{i}}^{\texttt{H},\pm},
\end{equation*}
where $\varepsilon^\rho_{\ell,\bm{i}}$ is a small positive number defined by $\varepsilon^\rho_{\ell,\bm{i}} := \min\{10^{-13},  \rho(\bU_{\ell,\bm{i}}^{\texttt{L},\pm}) \}$.
The blending parameter for the density is chosen as
\begin{equation*}%\label{eq:blending_parameter_rho}
	\theta^{\rho}_{\bm{i} + \frac12\bm{e}_{\ell}}
	=
	\min\left\{
	\Theta\left(
	\rho(\bU_{\ell,\bm{i}}^{\texttt{L},+}),
	\rho(\bU_{\ell,\bm{i}}^{\texttt{H},+}),
	\varepsilon^\rho_{\ell,\bm{i}}
	\right),
	\Theta\left(
	\rho(\bU_{\ell,\bm{i} + \bm{e}_{\ell}}^{\texttt{L},-}),
	\rho(\bU_{\ell,\bm{i} + \bm{e}_{\ell}}^{\texttt{H},-}),
	\varepsilon^\rho_{\ell,\bm{i} + \bm{e}_{\ell}}
	\right)
	\right\},
\end{equation*}
which considers the two cells sharing the same interface.

The second step is to impose the pressure positivity
\begin{equation*}
	p({\bU}_{\ell,\bm{i}}^{\texttt{Lim},\texttt{pp},+}) > \varepsilon^p_{\ell,\bm{i}}, \quad
	p({\bU}_{\ell,\bm{i}+\bm{e}_{\ell}}^{\texttt{Lim},\texttt{pp},-}) > \varepsilon^p_{\ell,\bm{i}+\bm{e}_{\ell}},
\end{equation*}
where $\varepsilon^p_{\ell,\bm{i}}$ is a small positive number defined by $\varepsilon^p_{\ell,\bm{i}} := \min\{10^{-13},  p(\bU_{\ell,\bm{i}}^{\texttt{L},\pm}) \}$.
The blending parameter for the pressure is chosen as
\begin{equation*}%\label{eq:blending_parameter_p}
	\theta^{p}_{\bm{i} + \frac12\bm{e}_{\ell}}
	=
	\min\left\{
	\Theta\left(
	p(\bU_{\ell,\bm{i}}^{\texttt{L},+}),
	p(\widetilde{\bU}_{\ell,\bm{i}}^{\texttt{Lim},\rho,+}),
	\varepsilon^p_{\ell,\bm{i}}
	\right),
	\Theta\left(
	p(\bU_{\ell,\bm{i} + \bm{e}_{\ell}}^{\texttt{L},-}),
	p(\widetilde{\bU}_{\ell,\bm{i} + \bm{e}_{\ell}}^{\texttt{Lim},\rho,-}),
	\varepsilon^p_{\ell,\bm{i} + \bm{e}_{\ell}}
	\right)
	\right\},
\end{equation*}
which uses the concavity of the pressure function.
Thus the final limited intermediate states and numerical flux are
\begin{align*}
	&{\bU}_{\ell,\bm{i}}^{\texttt{Lim},\texttt{pp},\pm}
	= (1 - \theta^\rho_{\bm{i} \pm \frac12\bm{e}_\ell}\theta^p_{\bm{i} \pm \frac12\bm{e}_\ell})\bU_{\ell,\bm{i}}^{\texttt{L},\pm}
	+ \theta^\rho_{\bm{i} \pm \frac12\bm{e}_\ell}\theta^p_{\bm{i} \pm \frac12\bm{e}_\ell} \bU_{\ell,\bm{i}}^{\texttt{H},\pm}, \\
	&\widehat{\bF}^{\texttt{Lim},\texttt{pp}}_{\bm{i} + \frac12\bm{e}_\ell} 
	= \widehat{\bF}^{\texttt{L}}_{\bm{i} + \frac12\bm{e}_\ell}
	+ \theta^\rho_{\bm{i} + \frac12\bm{e}_\ell}\theta^p_{\bm{i} + \frac12\bm{e}_\ell} \Delta\widehat{\bF}^{\texttt{Lim},\texttt{s}}_{\bm{i} + \frac12\bm{e}_\ell}.
\end{align*}

\begin{remark}
	Note that the cell-centered value should also be limited using the scaling limiter with the help of the cell average,
	see \cite{Duan_2025_Active_SJSC} for the details.
\end{remark}

\subsection{Limiting for the point value}
To achieve the PP property, a simple scaling limiter is adopted.
One can easily write down a low-order PP LLF method for the point value based on the spatial stencil used in \cite{Duan_2025_Active_SJSC} and the numerical flux defined in \eqref{eq:ns_low_face_fluxes}.
Then the high-order AF solution is blended with the low-order one to preserve positivity, with the blending parameters determined similarly to those for cell averages.

Let us summarize the main results of the proposed PP AF method.
\begin{theorem}\rm
	If the initial numerical solution $\overline{\bU}_{\bm{i}}^0, \bU_{\bm{\zeta}}^0\in\mathcal{G}$ for all $\bm{i}$ and $\bm{\zeta}$,
	and the time-step size satisfies \eqref{eq:ns_1d_monolithic_cfl} for the cell average and also the one for the point value,
	then the AF method \eqref{eq:af_av_1d}--\eqref{eq:af_pnt_1d} equipped with an SSP Runge--Kutta method
	and the PP limiters for cell averages \eqref{eq:limited_xl} and point values preserves density and pressure positivity.
\end{theorem}

\begin{remark}\label{rmk:ns_time_step}
	In the actual implementation, the time-step size is first chosen based on the CFL condition \eqref{eq:cfl}.
	If the low-order counterparts needed in the proposed limiting are not PP, then the solution is reset and the full RK step is recomputed with a halved time-step size.
\end{remark}

\section{Numerical experiments}\label{sec:results}

This section presents numerical experiments for the compressible NS equations \eqref{eq:ns},
including smooth accuracy tests, problems with strong shocks and low pressure, and viscous flows at low Mach numbers.
The ratio of specific heats is $\gamma=1.4$ unless specified otherwise.

\begin{example}[1D viscous shock wave]\label{ex:1d_accuracy}
	Consider the moving viscous shock used in \cite{Lin_2023_positivity_JCP} on the domain $[-1,1.5]$.
	Let $u_\infty$ denote the translation velocity and $\xi=x-u_\infty t$.
	Let $w(\xi)$ denote the velocity relative to the moving shock.
	The exact primitive variables are $\rho(\xi)$, $u(x,t)=u_\infty+w(\xi)$, and $p(\xi)$, where
	\begin{equation*}
		\rho(\xi)=\frac{m_0}{w(\xi)},\quad p(\xi)=(\gamma-1)\rho(\xi)e(\xi), \quad e(\xi)=\frac{1}{2\gamma}\left(\frac{\gamma+1}{\gamma-1}w_0^2-w(\xi)^2\right).
	\end{equation*}
	Let $w_L$ and $w_R$ be the velocities relative to the shock at $-\infty$ and $+\infty$, respectively.
	The velocity $w=w(\xi)$ is determined by
	\begin{equation*}
		\xi=\frac{2\mu\gamma}{(\gamma+1)m_0\Pr}\left[\frac{w_L}{w_L-w_R}\log\left(\frac{w_L-w}{w_L-w_0}\right)-\frac{w_R}{w_L-w_R}\log\left(\frac{w-w_R}{w_0-w_R}\right)\right],
	\end{equation*}
	which was derived by Becker \cite{Becker_1923_Stosswelle_P}.
	Here, $m_0=\rho_Lw_L$, $w_0=\sqrt{w_Lw_R}$.
	The ratio of the upstream and downstream relative velocities is determined by the upstream Mach number in the shock frame,
	\begin{equation*}
		\frac{w_R}{w_L}=\frac{(\gamma-1)+2/M_0^2}{\gamma+1},\quad M_0=\frac{w_L}{\sqrt{\gamma p_L/\rho_L}}.
	\end{equation*}
	Take $\Pr=3/4$,	$u_\infty=0.2$, $\rho_L=w_L=m_0=1$, and the initial shock is located at $x=0$.
	For the first case, take $M_0=3$ and $\mu=10^{-2}$ for the accuracy test without limiting,
	whereas the second case uses $M_0=20$ and $\mu=10^{-3}$ for the accuracy test with limiting.
	Both cases are evolved to $t_{\mathrm{final}}=0.2$.
	Boundary values at both ends of the domain are prescribed using the time-dependent exact solution.
	
	Figures~\ref{fig:1d_accuracy_mach3} and \ref{fig:1d_accuracy_mach20} show the numerical solutions and the $\ell_1$ errors and convergence rates for the first and second cases, respectively.
	The AF method accurately captures the exact solution
	and achieves fourth-order convergence under mesh refinement.
	
	\begin{figure}[htb!]
		\centering
		\includegraphics[width=0.48\textwidth]{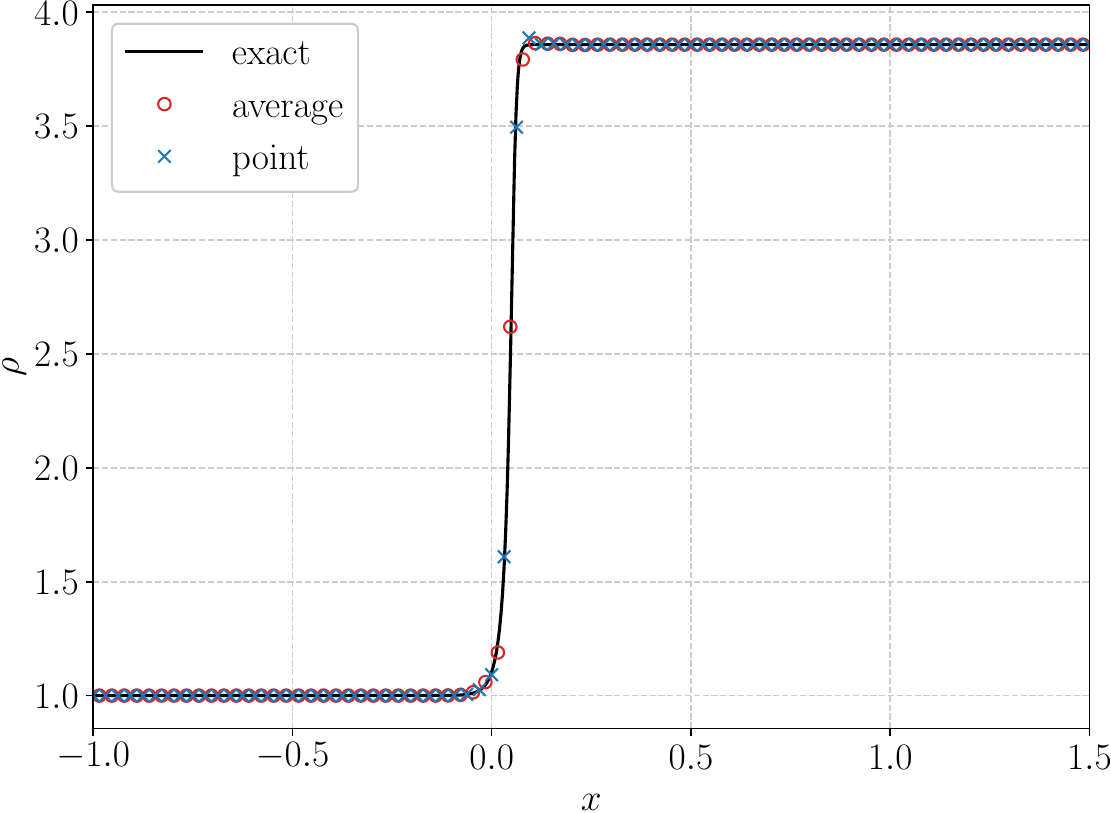}~
		\includegraphics[width=0.475\textwidth]{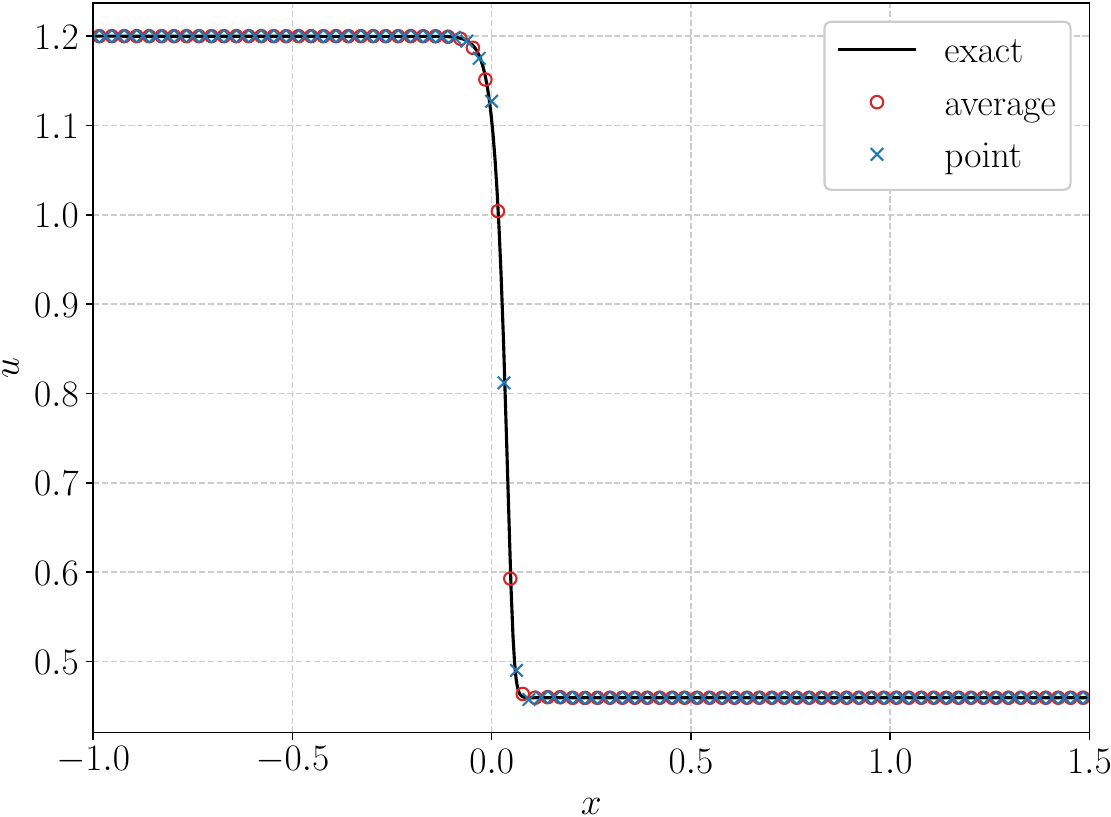}
		
		\vspace{2pt}
		
		\includegraphics[width=0.48\textwidth]{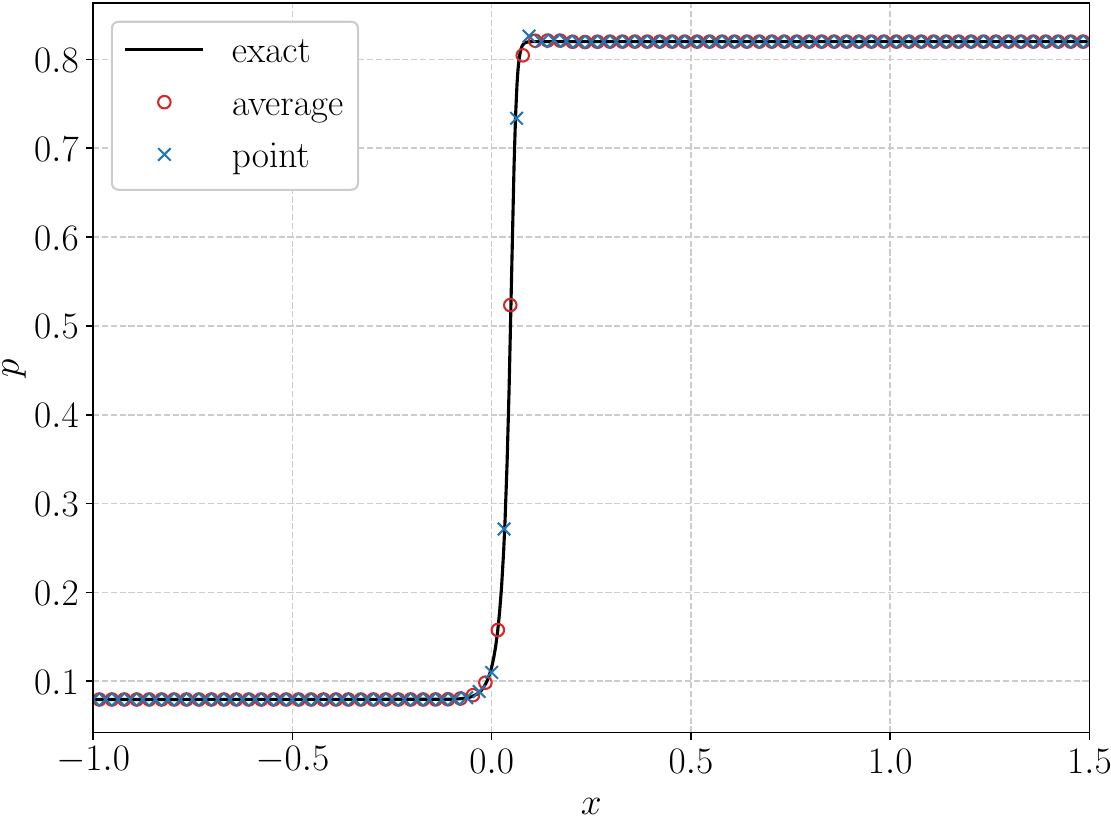}~
		\includegraphics[width=0.485\textwidth]{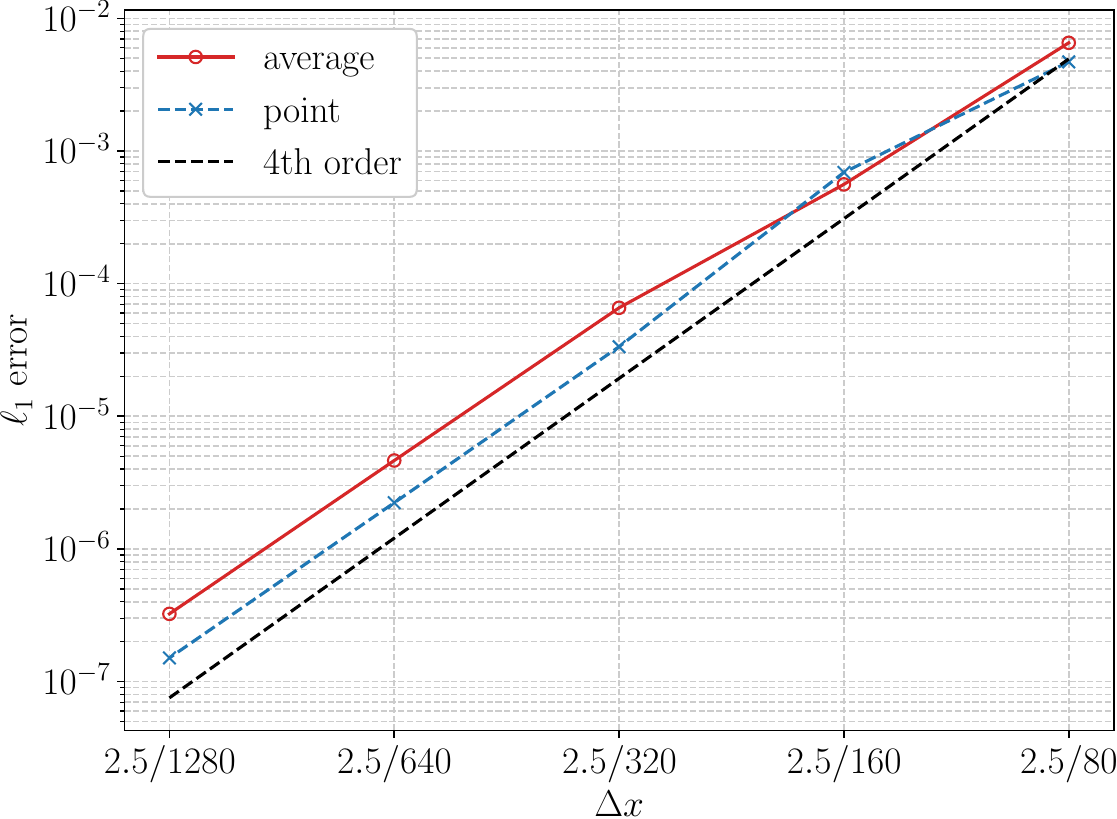}
		\caption{Example~\ref{ex:1d_accuracy}: viscous shock with $M_0=3$.
			The numerical solutions with $80$ cells and the $\ell_1$ errors and convergence rates of the AF method.}
		\label{fig:1d_accuracy_mach3}
	\end{figure}
	
	\begin{figure}[htb!]
		\centering
		\includegraphics[width=0.48\textwidth]{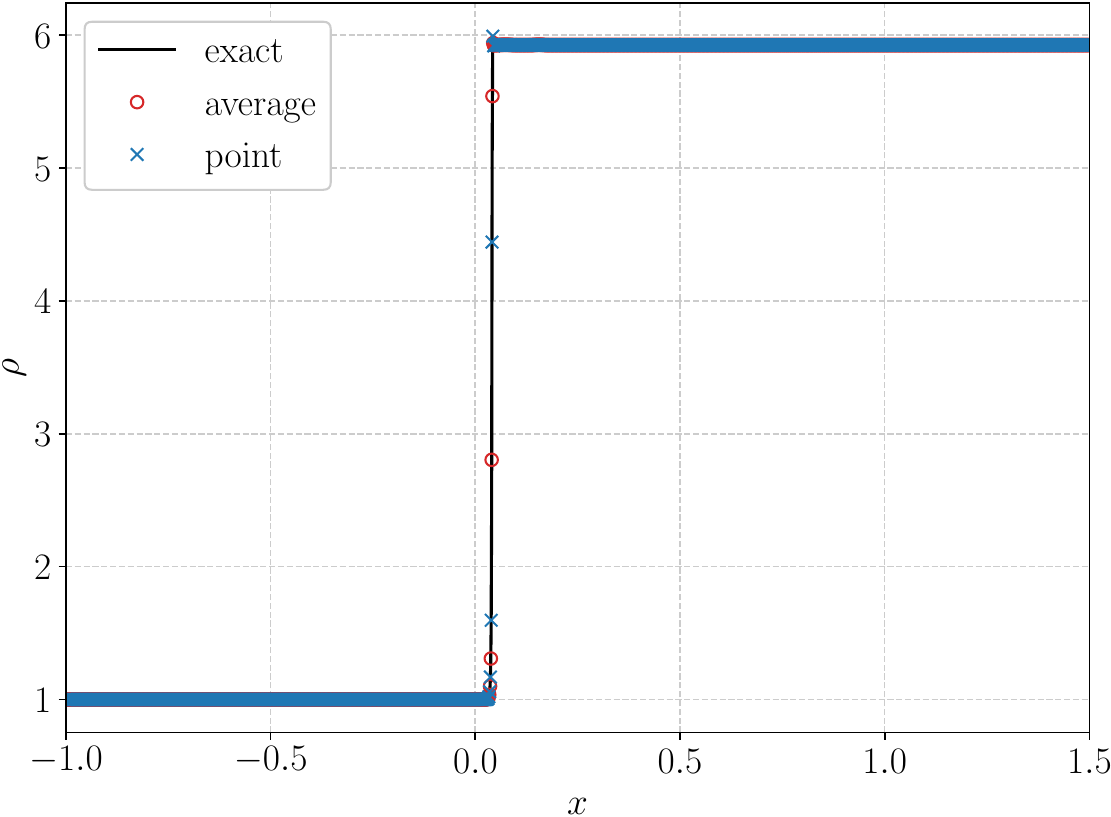}~
		\includegraphics[width=0.475\textwidth]{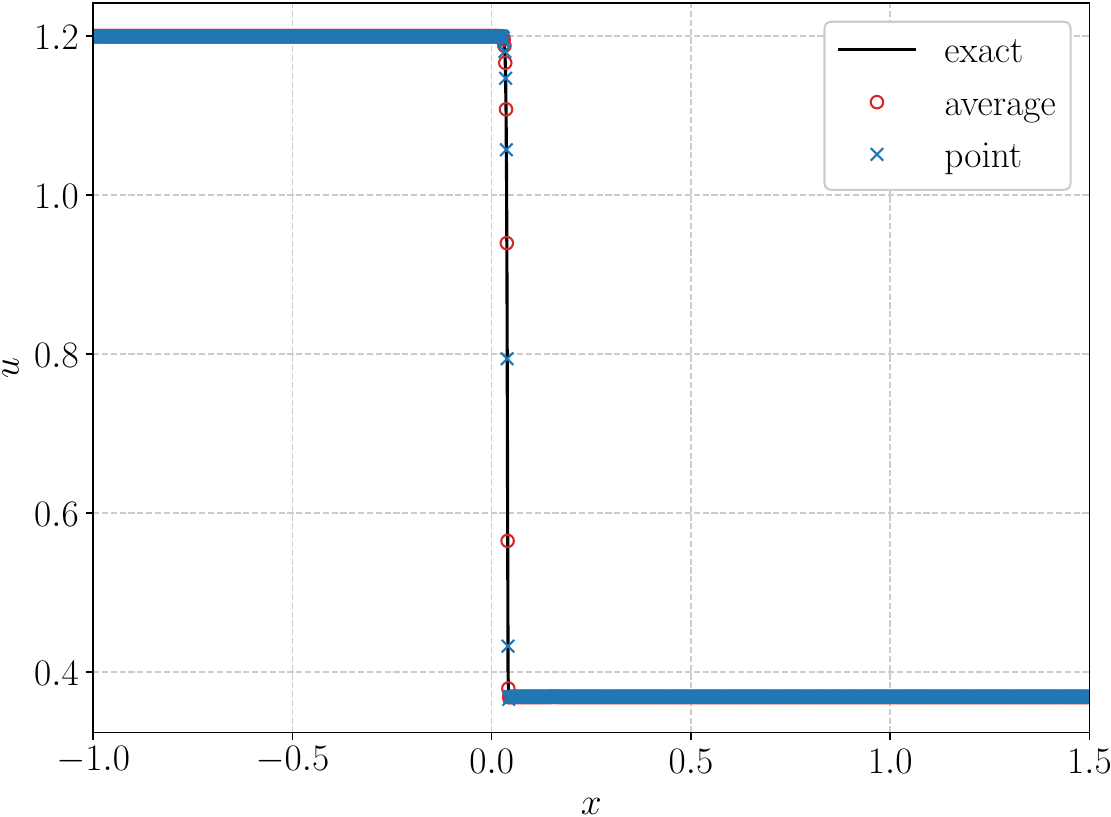}
		
		\vspace{2pt}
		
		\includegraphics[width=0.48\textwidth]{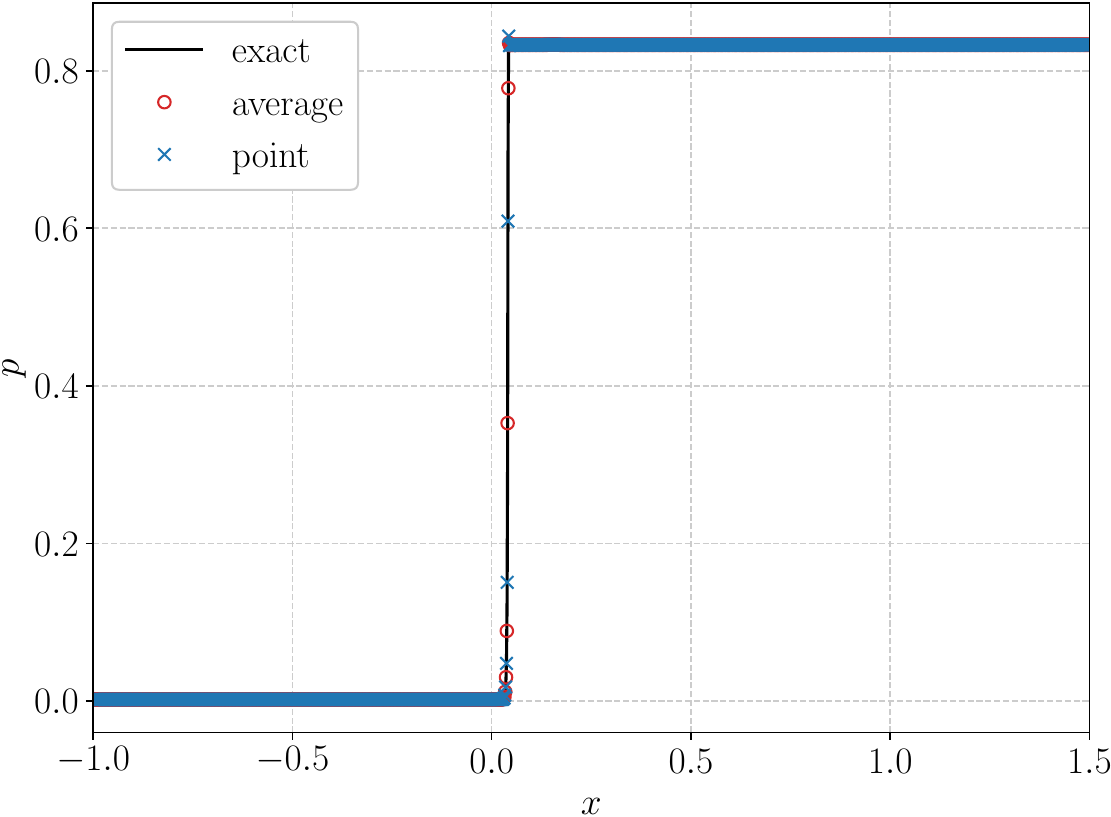}~
		\includegraphics[width=0.485\textwidth]{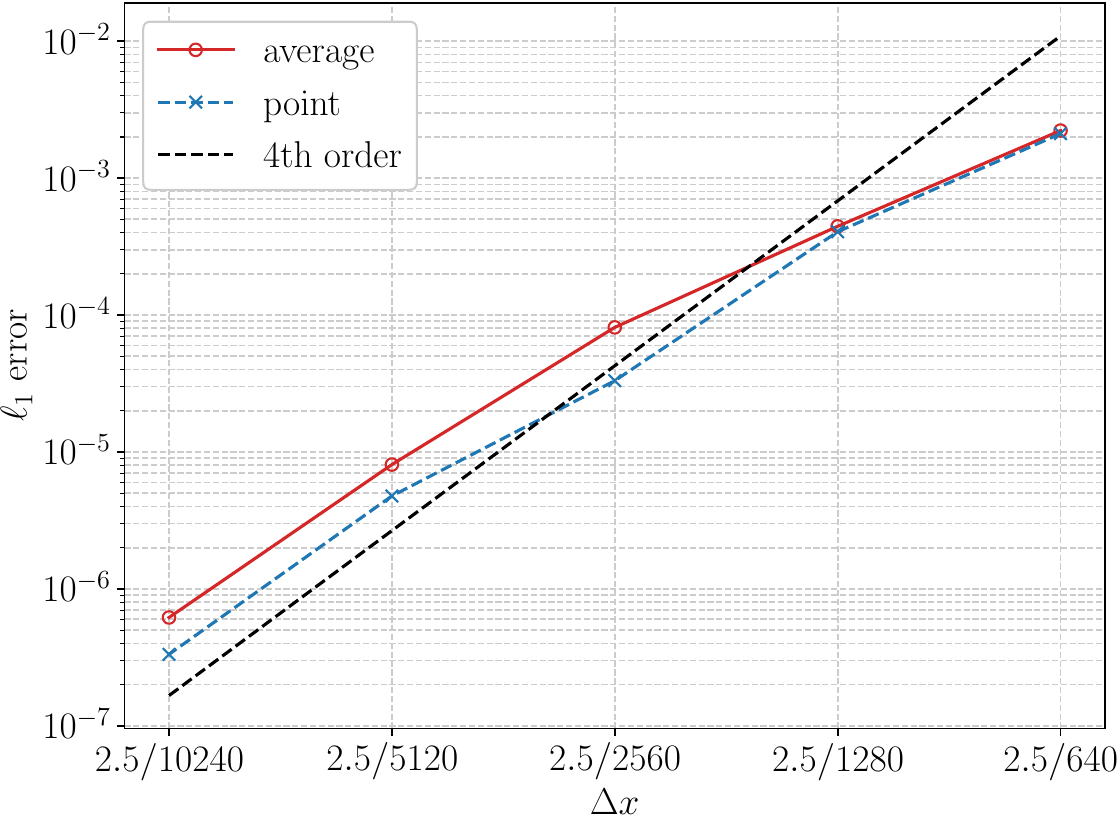}
		\caption{Example~\ref{ex:1d_accuracy}: viscous shock with $M_0=20$.
			The numerical solutions with $1280$ cells and the $\ell_1$ errors and convergence rates of the AF method.}
		\label{fig:1d_accuracy_mach20}
	\end{figure}
\end{example}

\begin{example}[2D manufactured Taylor--Green vortex]\label{ex:2d_taylor_green}
	Consider the periodic domain $[0,2\pi]\times[0,2\pi]$.
	Let $a(t)=\exp(-2\mu t)$ and $b(x,y)=\cos(2x)+\cos(2y)$.
	The prescribed solution is
	\begin{equation*}
		\rho=1,~
		\bu=\left(a\sin x\cos y\\-a\cos x\sin y\right)^\top,~
		p=\frac{p_0}{\gamma-1}+\frac{a^2b}{4},
	\end{equation*}
	with $p_0=10^2$.
	The mass and momentum equations are satisfied automatically.
	For the energy equation, one has to add the following manufactured energy source
	\begin{equation*}
		S_E=\frac{p_t+u p_x+v p_y}{\gamma-1}
		-4\mu a^2\cos^2x\cos^2y-\kappa\Delta p,
	\end{equation*}
	where $p_t=-\mu a^2b$, $p_x=-a^2\sin(2x)/2$,
	$p_y=-a^2\sin(2y)/2$, $\Delta p=-a^2b$.
	The viscosity is $\mu=10^{-5}$,	with $\Pr=1$ and $t_{\texttt{final}}=0.2$.
	
	The numerical solutions obtained by the proposed AF method with $80\times80$ cells are shown in Figure~\ref{fig:2d_taylor_green}.
	It is observed that the results agree well with the exact solution.
	The errors and convergence rates are presented in Figure~\ref{fig:2d_taylor_green_error},
	which clearly show fourth-order convergence.
	
	\begin{figure}[htb!]
		\centering
		\includegraphics[height=0.42\textwidth]{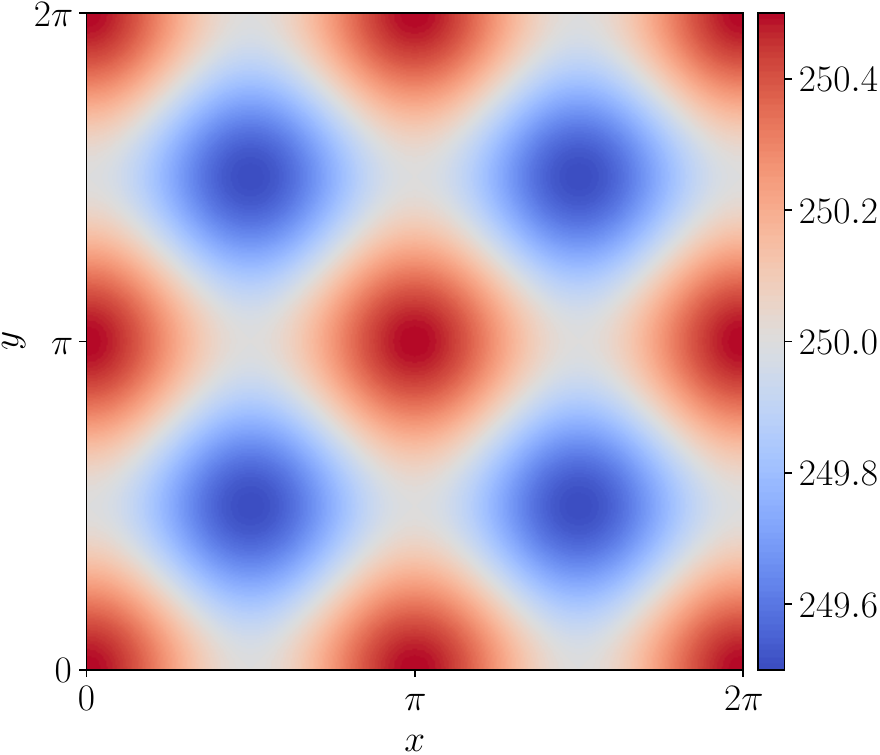}~
		\includegraphics[height=0.42\textwidth]{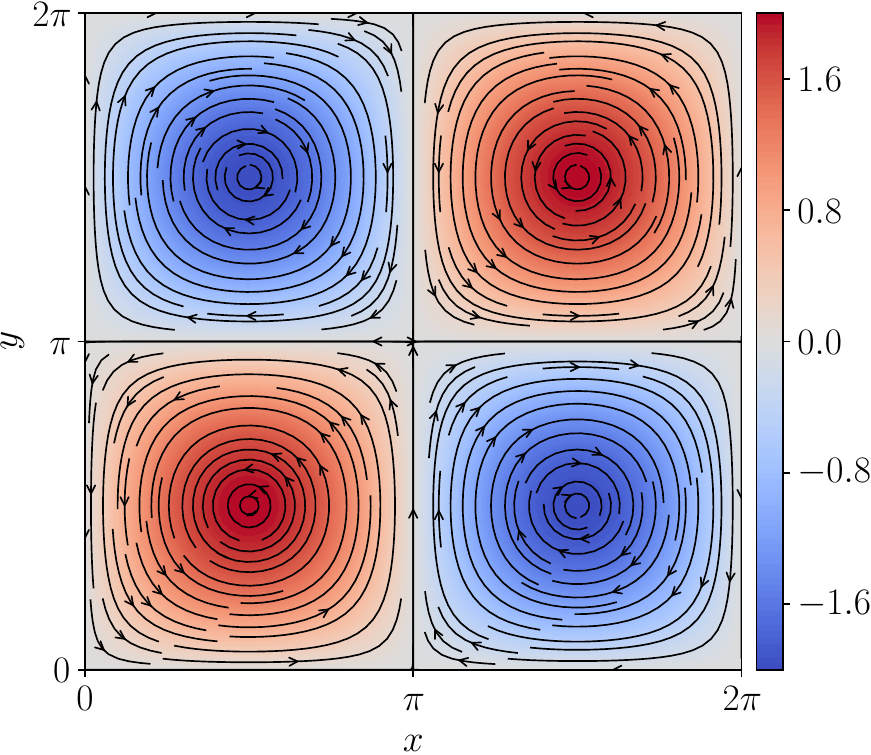}
		
		\vspace{2pt}
		
		\includegraphics[width=0.48\textwidth]{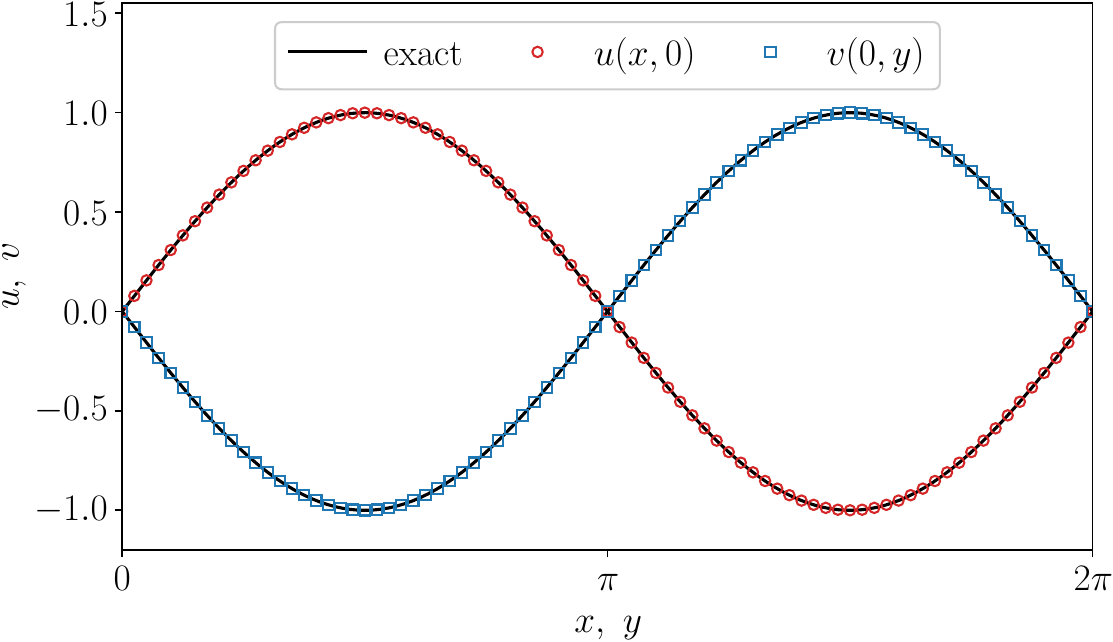}
		\includegraphics[width=0.48\textwidth]{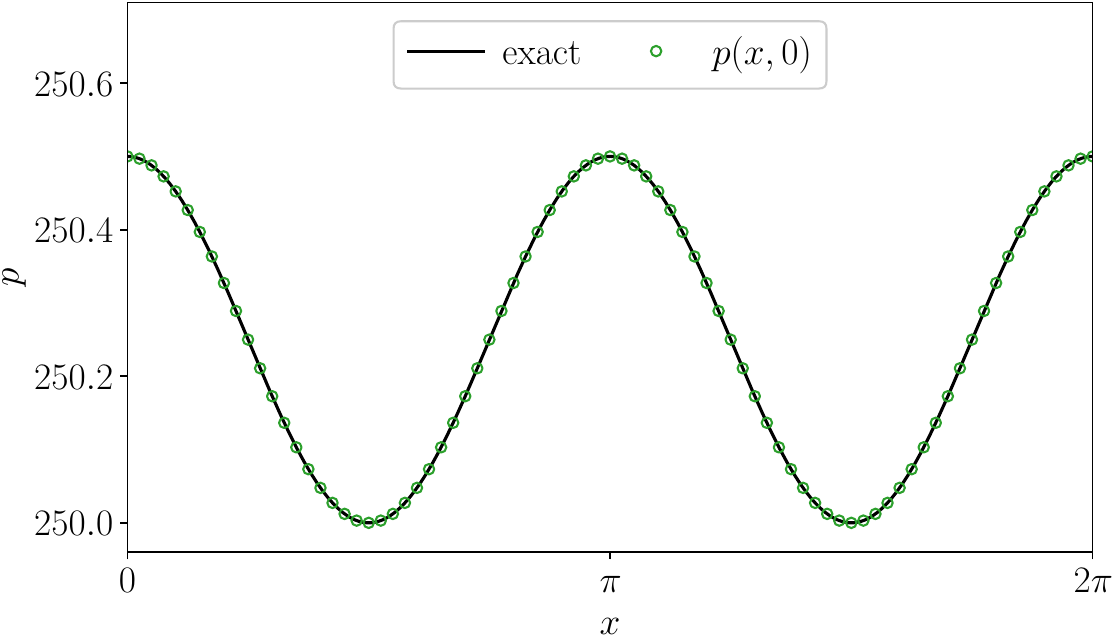}
		\caption{Example~\ref{ex:2d_taylor_green}: manufactured Taylor--Green vortex.
			The pressure, streamlines colored by the vorticity, velocity cut-lines, and pressure cut-line (from upper left to lower right).
		}
		\label{fig:2d_taylor_green}
	\end{figure}
	
	\begin{figure}[htb!]
		\centering
		\includegraphics[width=0.45\textwidth]{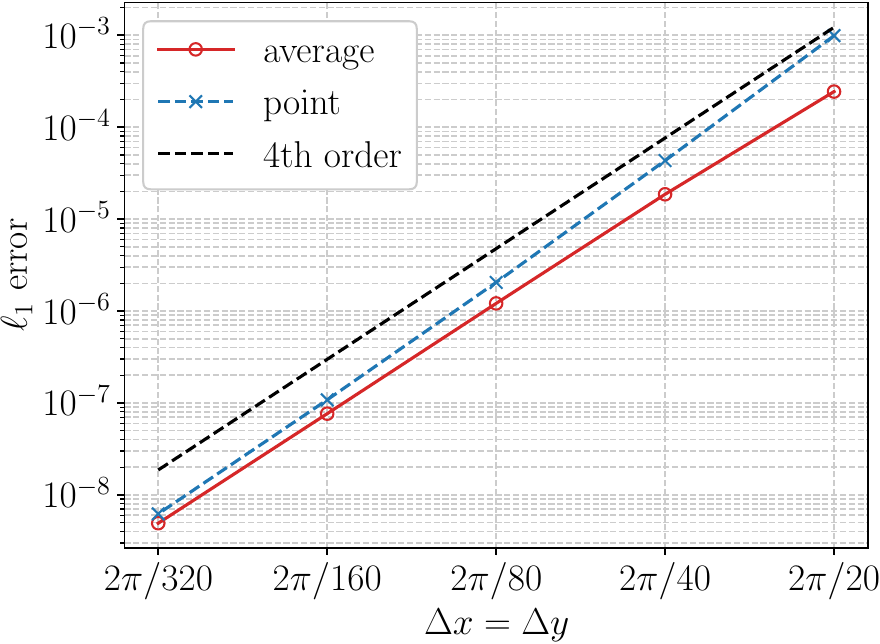}
		\caption{Example~\ref{ex:2d_taylor_green}: manufactured Taylor--Green vortex.
			The errors and convergence rates.}
		\label{fig:2d_taylor_green_error}
	\end{figure}
\end{example}

\begin{example}[First problem of Stokes]\label{ex:1d_stokes}
	This test considers the diffusion of a tangential velocity discontinuity,
	where the solution depends only on $x$, while the transverse velocity is retained.
	The computational domain is $[-0.5,0.5]$ with outflow boundary conditions, and the initial condition is
	\begin{equation*}
		(\rho,u,v,p) =
		\begin{cases}
			(1, ~0, ~-0.1, ~1/\gamma),&\text{if}~ x<0,\\
			(1, ~0, ~0.1, ~1/\gamma),&\text{otherwise}.
		\end{cases}
	\end{equation*}
	We consider $\mu=10^{-2}$, $10^{-3}$, and $10^{-4}$, with $\kappa=0$ in all three cases.
	The test is solved until $t_{\texttt{final}}=1$.
	In the incompressible limit, this test has an analytical solution given by $v = \frac{1}{10}\text{erf}\left(\frac{x}{2\sqrt{\mu t}}\right)$.
	
	Figure~\ref{fig:1d_stokes} presents the numerical solutions for the three viscosity values $\mu=10^{-2}$, $10^{-3}$, and $10^{-4}$ using $100$ cells.
	The results agree well with the reference solutions without oscillations.
	
	\begin{figure}[htb!]
		\centering
		\includegraphics[width=0.32\textwidth]{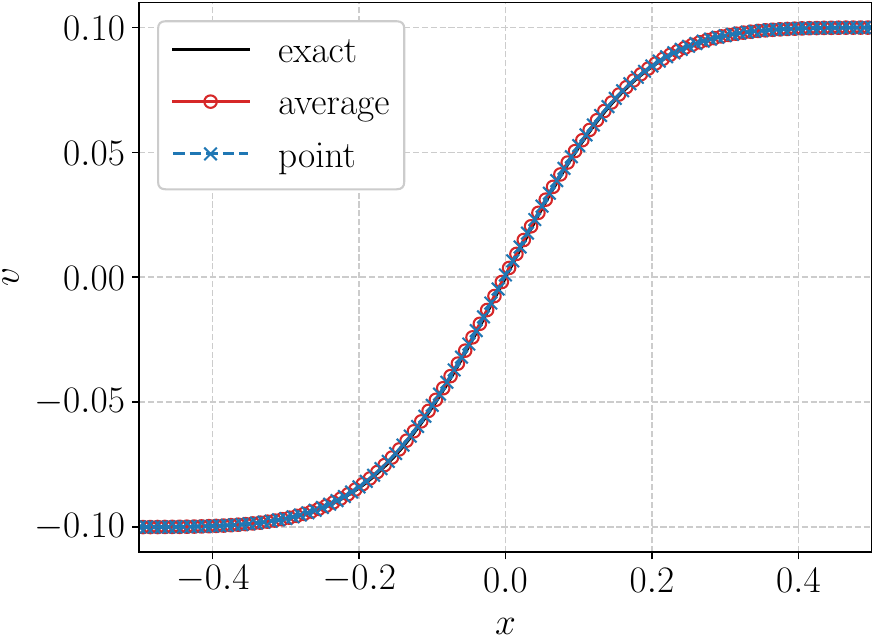}~
		\includegraphics[width=0.32\textwidth]{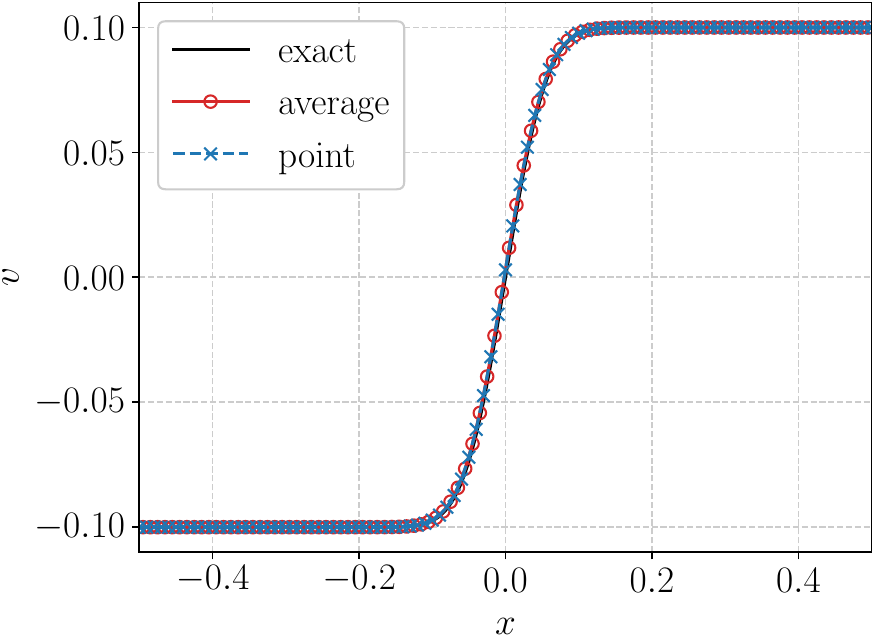}~
		\includegraphics[width=0.32\textwidth]{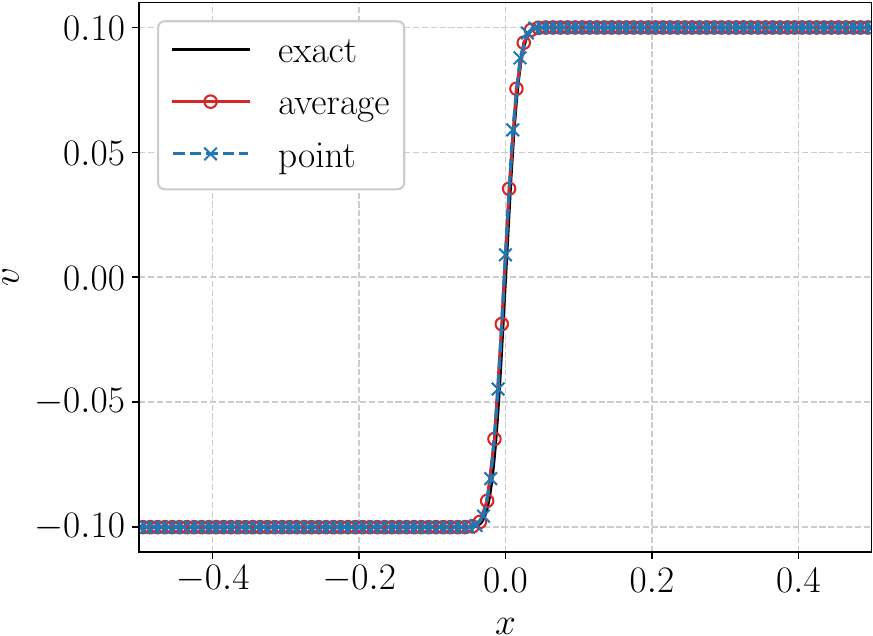}
		\caption{Example~\ref{ex:1d_stokes}: the first problem of Stokes.
			The tangential velocity obtained by the AF method with $\mu=10^{-2}$, $10^{-3}$, $10^{-4}$ (from left to right).}
		\label{fig:1d_stokes}
	\end{figure}
\end{example}

\begin{example}[One-dimensional LeBlanc problem]\label{ex:1d_leblanc}
	We consider the one-dimensional LeBlanc problem using the initial data
	associated with the test in \cite{Upperman_2022_Positivity_JCP}.
	On the domain $[-0.5,0.5]$, the initial condition is
	\begin{equation*}
		(\rho,u,p)(x,0)=
		\begin{cases}
			(1,0,6.666667\times10^{-2}),&x<-0.2,\\
			(10^{-2},0,6.666667\times10^{-11}),&x\geq-0.2.
		\end{cases}
	\end{equation*}
	We consider the inviscid case $\mu=0$ and the viscous case $\mu=10^{-5}$ with $\Pr=0.75$.
	Both cases use outflow boundary conditions and are evolved to $t_{\mathrm{final}}=0.4$.
	
	The numerical solutions obtained with $800$ cells are shown in Figure~\ref{fig:1d_leblanc}.
	The inviscid results are compared with the exact Euler solution, and the viscous results are compared with an AF reference solution computed using $2000$ cells.
	Viscosity smooths the discontinuities.
	A small density undershoot appears near the tail of the rarefaction in the inviscid case;
	otherwise, the numerical solutions agree well with the reference solutions.
	
	\begin{figure}[htb!]
		\centering
		\includegraphics[width=0.48\textwidth]{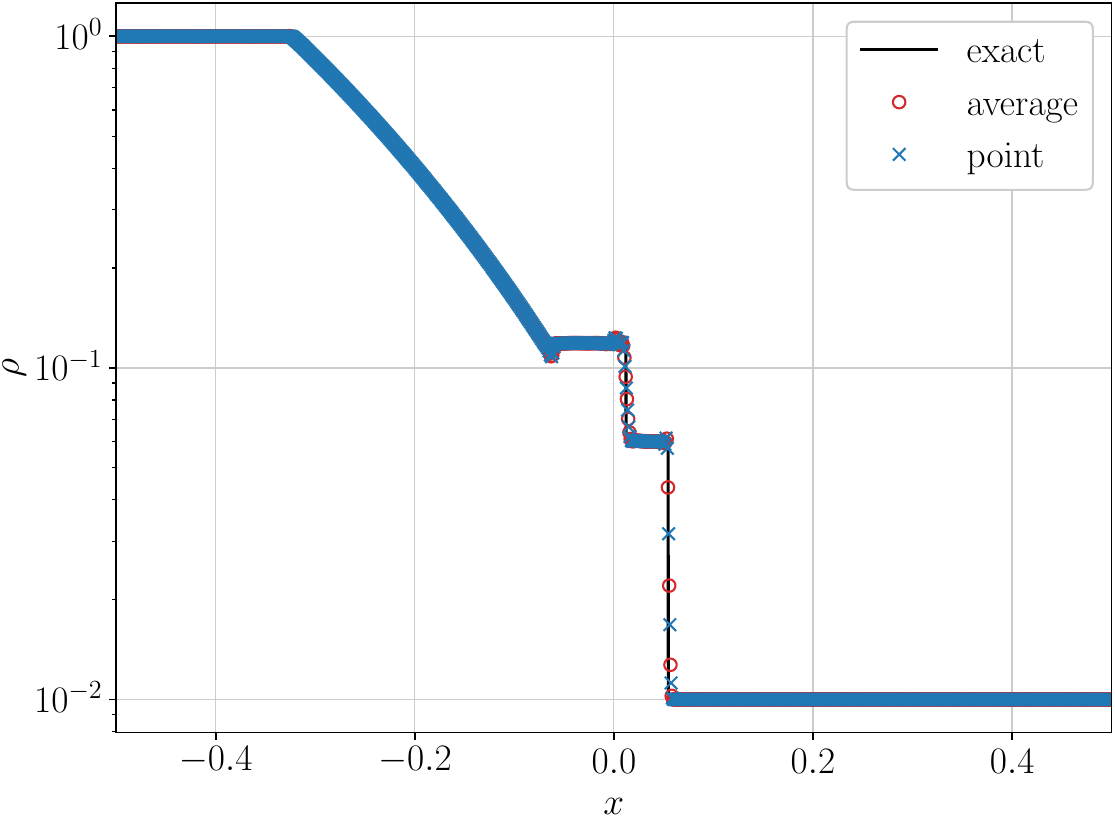}~
		\includegraphics[width=0.475\textwidth]{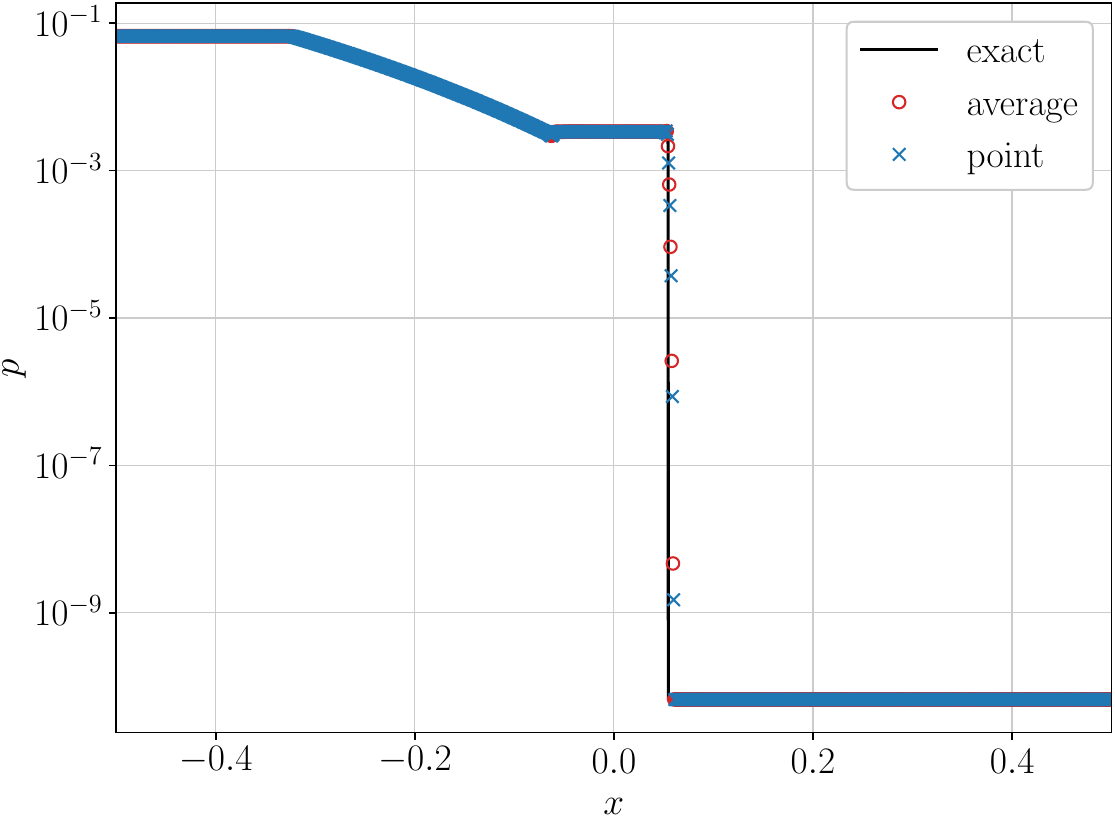}
		
		\vspace{2pt}
		
		\includegraphics[width=0.48\textwidth]{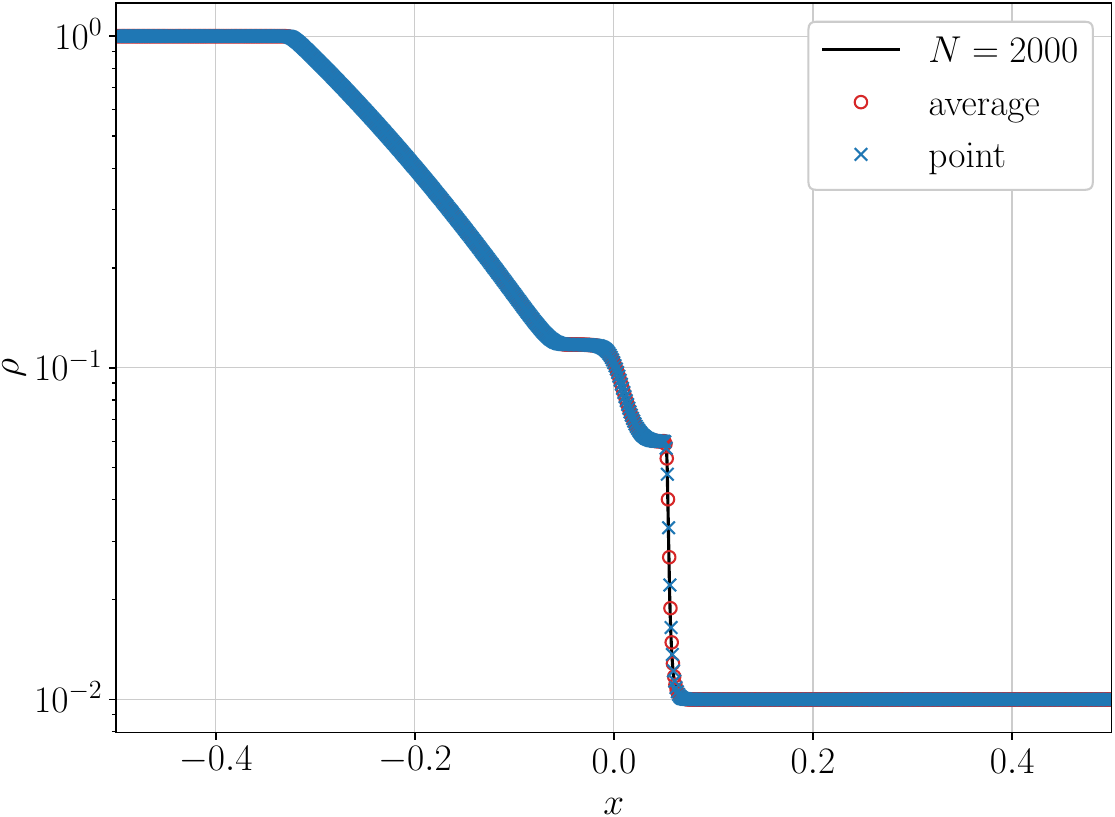}~
		\includegraphics[width=0.485\textwidth]{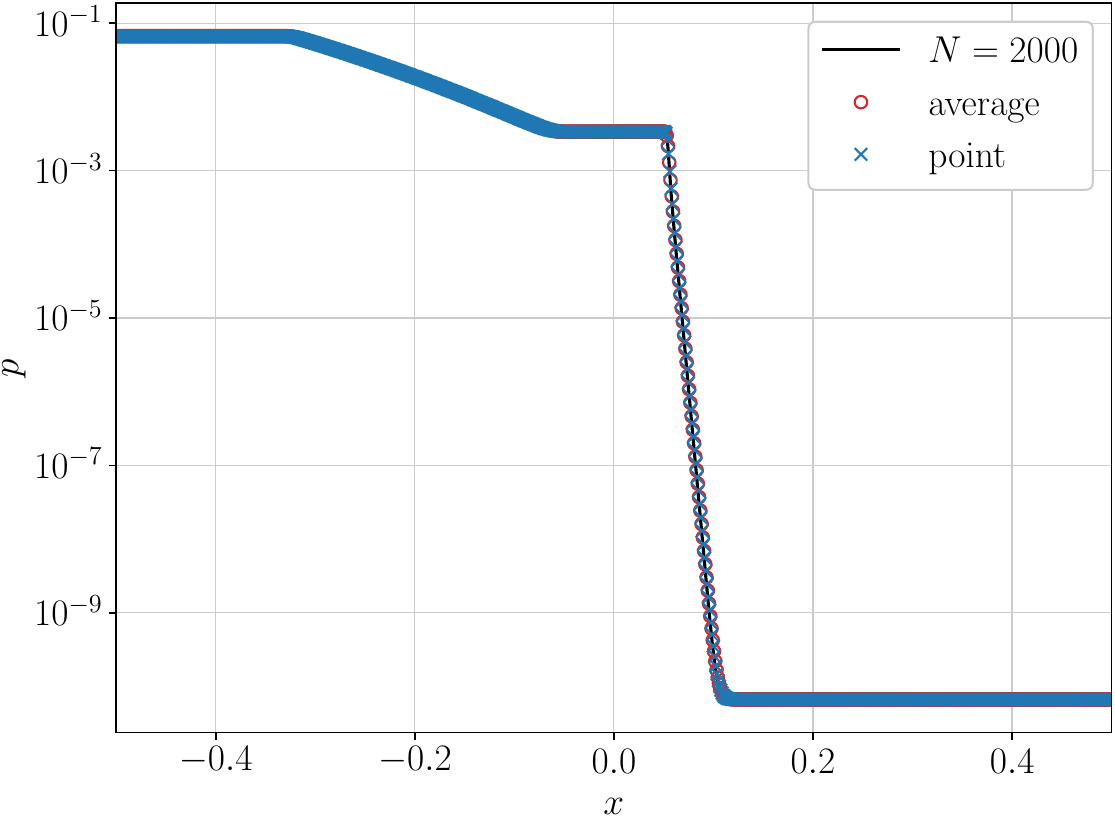}
		\caption{Example~\ref{ex:1d_leblanc}: LeBlanc test.
			The numerical solutions obtained with $800$ cells for $\mu=0$ (upper) and $\mu=10^{-5}$ (lower).}
		\label{fig:1d_leblanc}
	\end{figure}
\end{example}

\begin{example}[Shock diffraction]\label{ex:shock_diffraction}
	Consider the diffraction of a Mach $5.09$ shock around a corner, following \cite{Zhang_2017_positivity_JCP}.
	The computational domain is $[0,13]\times[0,11]$ with the solid region $[0,1]\times[0,6]$ excluded.
	Slip adiabatic wall conditions are imposed on the exposed solid boundaries.
	The post-shock state is prescribed at the left inlet for $y\geqslant6$, and outflow conditions are imposed on the remaining external boundaries.
	Initially, the right-moving shock is located at $x=0.5$.
	The initial condition is
	\begin{equation*}
		(\rho,u,v,p)=
		\begin{cases}
			(1.4,0,0,1), &\text{if}~x > 0.5, \\
			(7.041132907,4.077946955,0,30.05945), &\text{otherwise}. \\
		\end{cases}
	\end{equation*}
	We solve the inviscid Euler equations with $\mu=0$ and the NS equations with $\mathrm{Re}=200$ and $\Pr=0.72$. Both cases are evolved to $t_{\texttt{final}}=2.3$.
	
	The numerical solutions obtained by the proposed AF method are shown in Figure~\ref{fig:shock_diffraction}.
	Our results agree well with those in \cite{Zhang_2017_positivity_JCP}, and the viscosity makes the solution much smoother near the corner.
	
	\begin{figure}[htb!]
		\centering
		\includegraphics[width=0.48\textwidth]{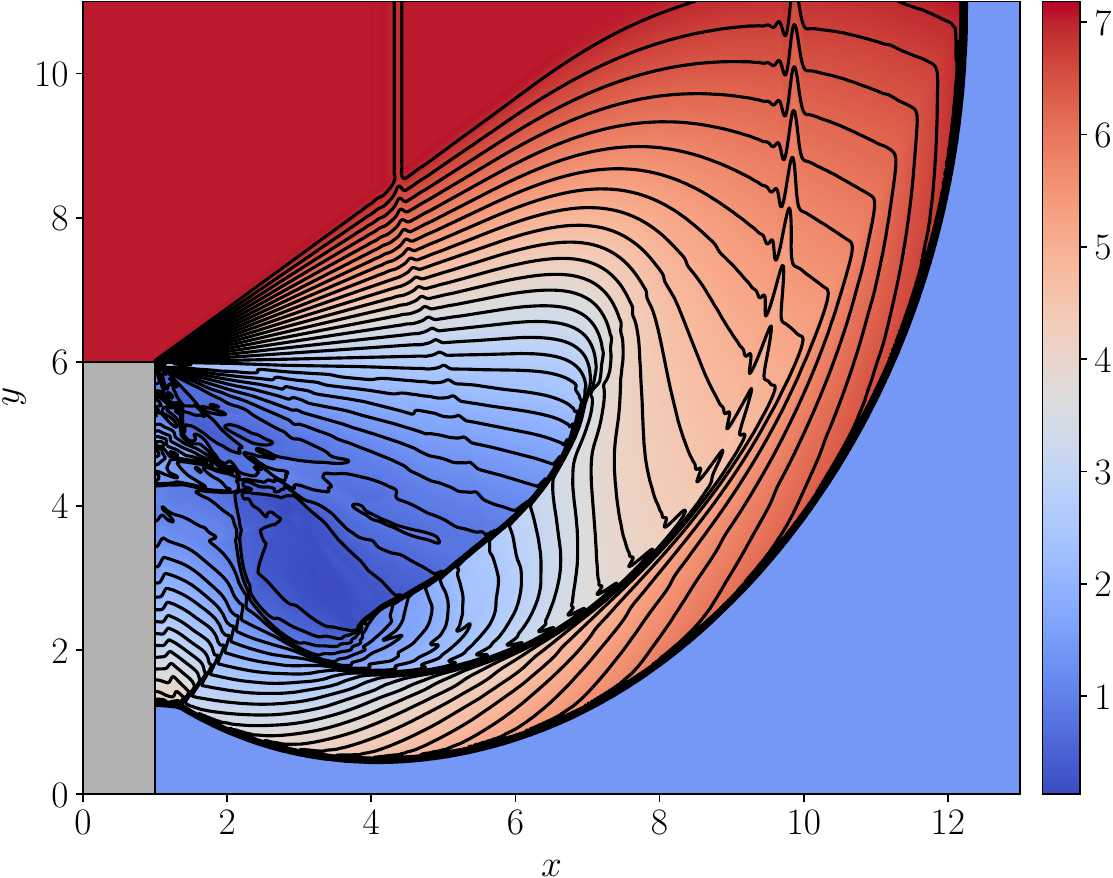}~
		\includegraphics[width=0.48\textwidth]{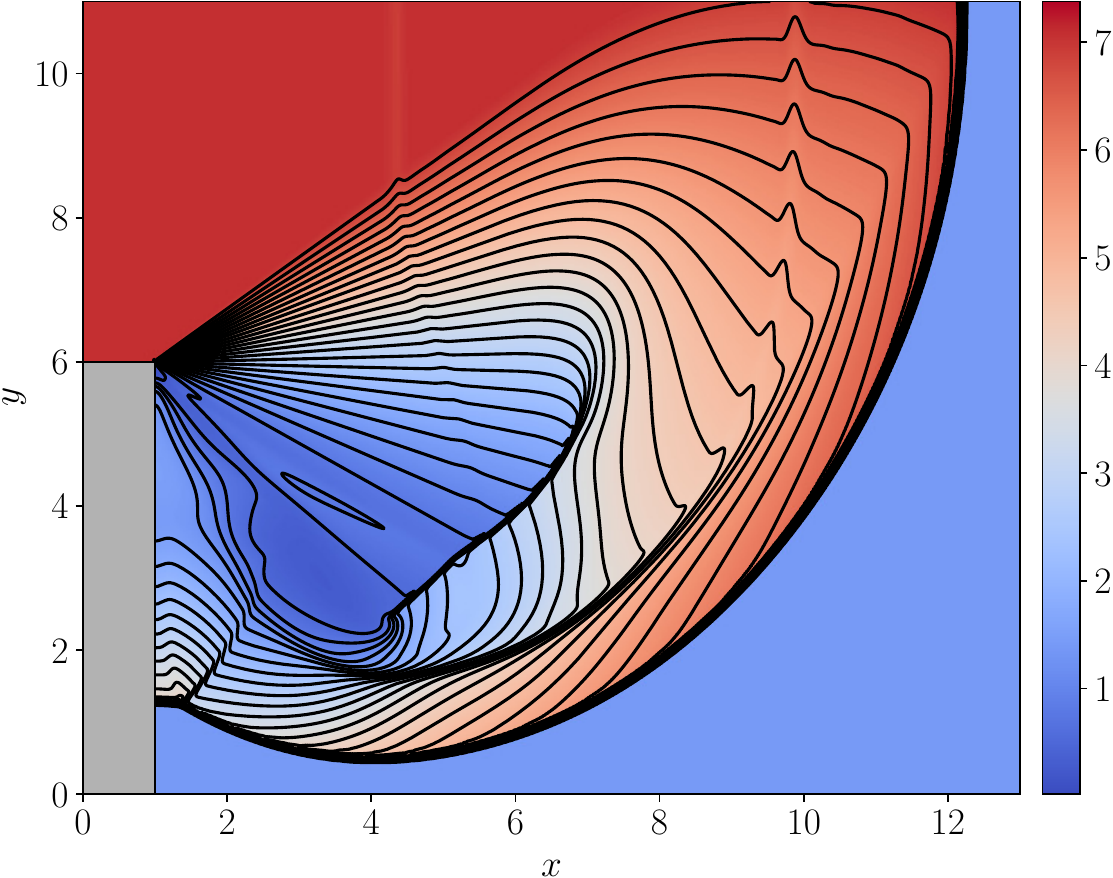}
		\caption{Example~\ref{ex:shock_diffraction}: shock diffraction.
			The density obtained  on a $520\times440$ mesh for $\mu=0$ (left) and $\mathrm{Re}=200$ (right).}
		\label{fig:shock_diffraction}
	\end{figure}
	
\end{example}

\begin{example}[Two-dimensional viscous shock tube with shock--boundary-layer interaction]\label{ex:2d_viscous_shock_tube}
	To examine the interaction of shocks with boundary layers, consider the shock tube problem in \cite{Daru_2001_Evaluation_CF}.
	The computational domain is the lower half of the square tube $[0,1]\times[0,0.5]$,
	with a symmetry boundary condition for the top boundary,
	while no-slip conditions are imposed elsewhere.
	The initial condition is
	\begin{equation*}
		(\rho,u,v,p)=
		\begin{cases}
			(120,0,0,120/\gamma),&\text{if}~ x<0.5,\\
			(1.2,0,0,1.2/\gamma),&\text{otherwise},
		\end{cases}
	\end{equation*}
	and $\mu=10^{-3}$ and $\Pr=0.73$.
	The output time is $t_{\texttt{final}}=1$.
	The reflected shock wave interacts with the boundary layers,
	leading to complex flow patterns consisting of a lambda shock, a main vortex, and other features.
	
	Let $g = |\nabla\rho|$ be the magnitude of the density gradient.
	Figure~\ref{fig:2d_viscous_shock_tube} shows the indicator $\phi = \exp(-10(g - g_{\min})/(g_{\max} - g_{\min}))$,
	obtained by the AF method and the $Q^3$ entropy stable DG method implemented in the Julia code of \cite{Lin_2023_positivity_JCP}.
	For the DG results shown here, additional shock capturing is disabled, while positivity limiting is retained.
	Smaller $\phi$ (red) indicates a larger normalized density gradient, while larger $\phi$ (blue) indicates a smaller one.
	It is observed that the proposed AF method can capture the shock waves, vortices, and other small features with good resolution.
	When using the same total number of DoFs, the AF method generally captures the flow structures better than the DG method.
	For example, comparing the results obtained by the AF method with $800\times400$ cells and by the DG method with $400\times200$ cells, the former captures the shock waves emerging from the lambda shock better,
	and also the structures between the wedge-shaped region located in $[0.5,0.7]\times[0,0.05]$.
	The results obtained by the AF method also agree better with the reference solution provided in \cite{Zhou_2018_Grid_PF},
	which is obtained by a second-order gas kinetic scheme on a $5000\times2500$ mesh.
	
	The blending parameter of the proposed limiting for the AF is shown in Figure~\ref{fig:2d_viscous_shock_tube},
	i.e., $1 - \theta^{\texttt{pp}}\theta^{\texttt{s}}$.
	The plot indicates that the limiting is mainly activated near strong discontinuities, thus the AF method can give high-resolution solutions.
	
	\begin{figure}[htb!]
		\centering
		\begin{subfigure}{0.49\textwidth}
			\includegraphics[width=\textwidth]{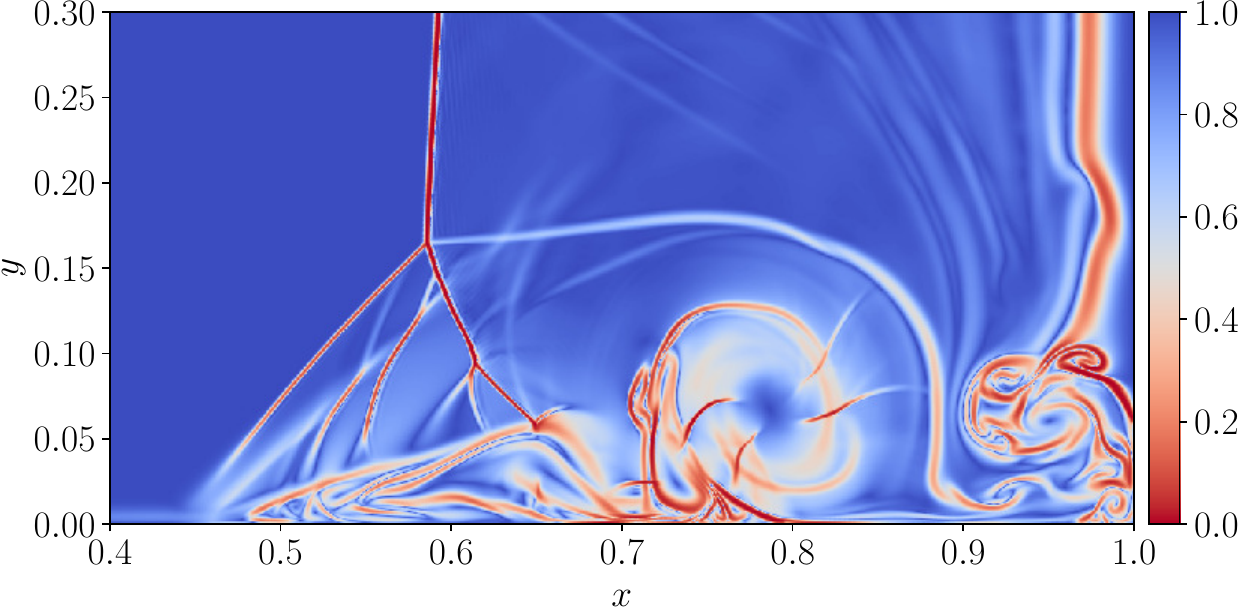}
			\caption{$\phi$ obtained by the AF with $800\times400$ cells}
		\end{subfigure}~
		\begin{subfigure}{0.49\textwidth}
			\includegraphics[width=\textwidth]{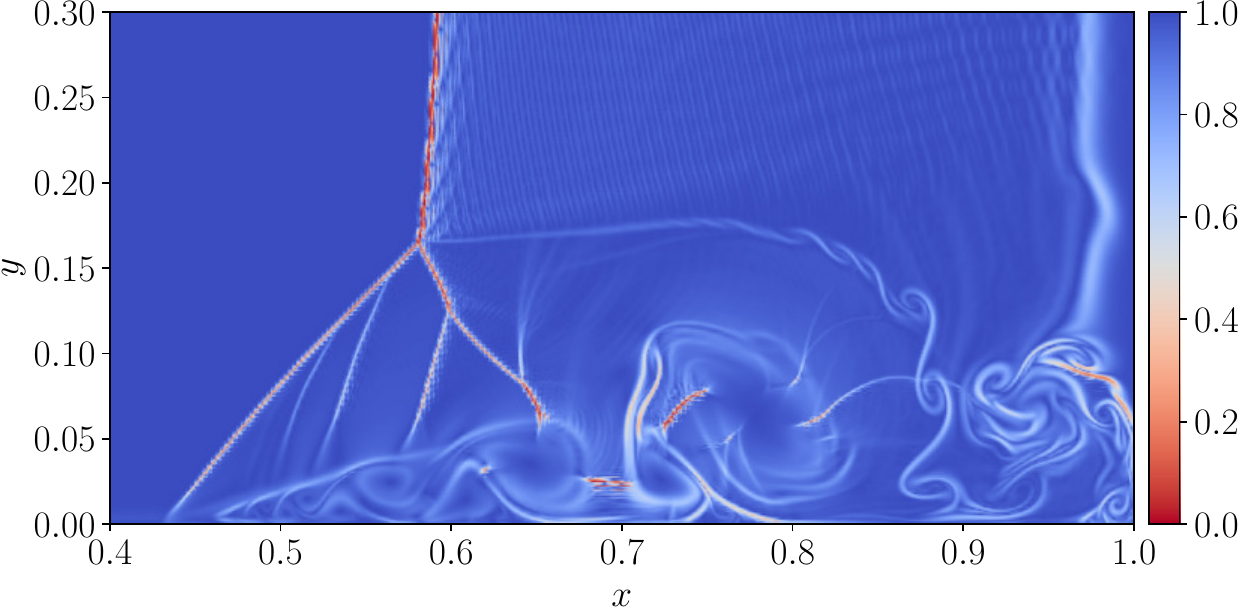}
			\caption{$\phi$ obtained by the DG with $400\times200$ cells}
		\end{subfigure}
		
		\vspace{5pt}
		
		\begin{subfigure}{0.49\textwidth}
			\includegraphics[width=\textwidth]{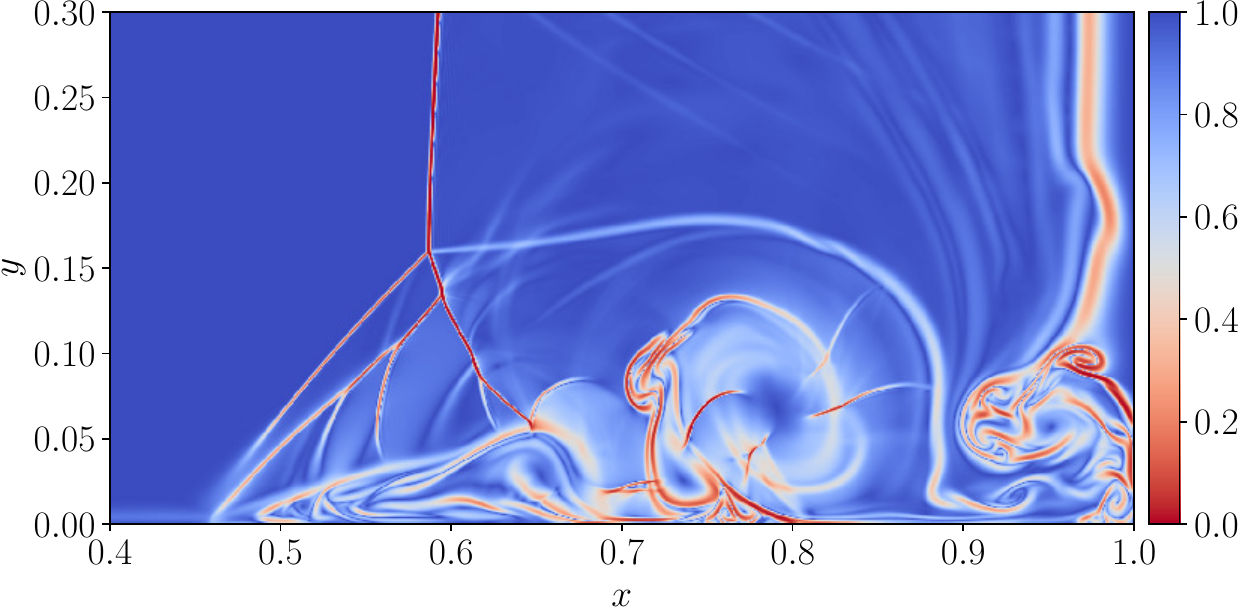}
			\caption{$\phi$ obtained by the AF with $1200\times600$ cells}
		\end{subfigure}~
		\begin{subfigure}{0.49\textwidth}
			\includegraphics[width=\textwidth]{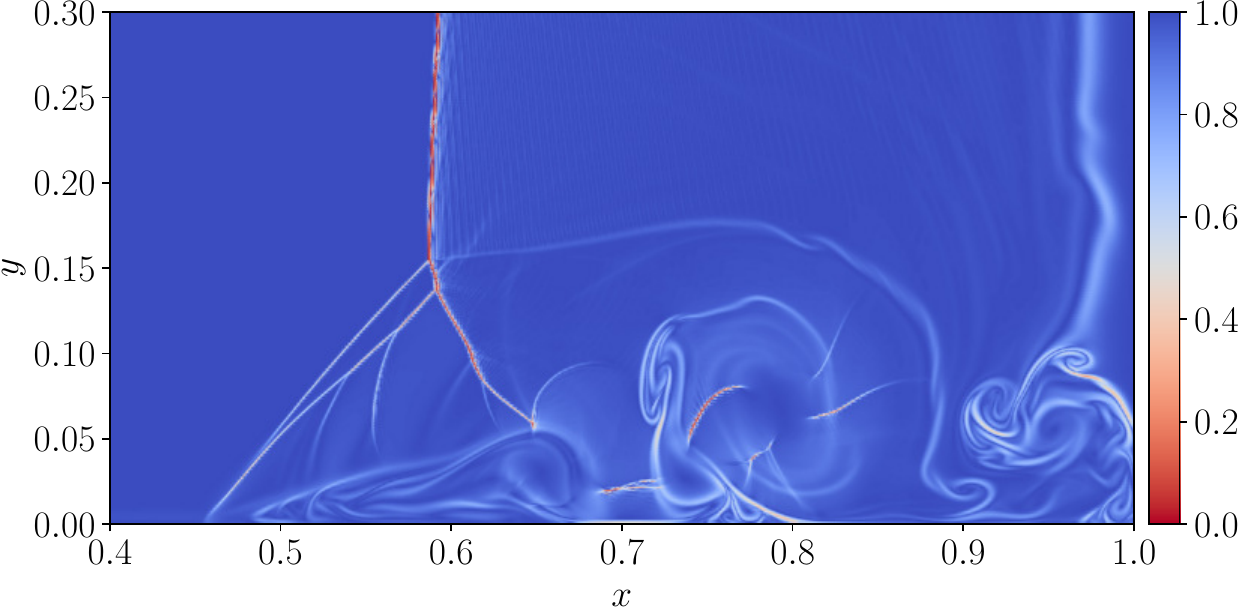}
			\caption{$\phi$ obtained by the DG with $600\times300$ cells}
		\end{subfigure}
		
		\vspace{5pt}
		
		\begin{subfigure}{0.49\textwidth}
			\includegraphics[width=\textwidth]{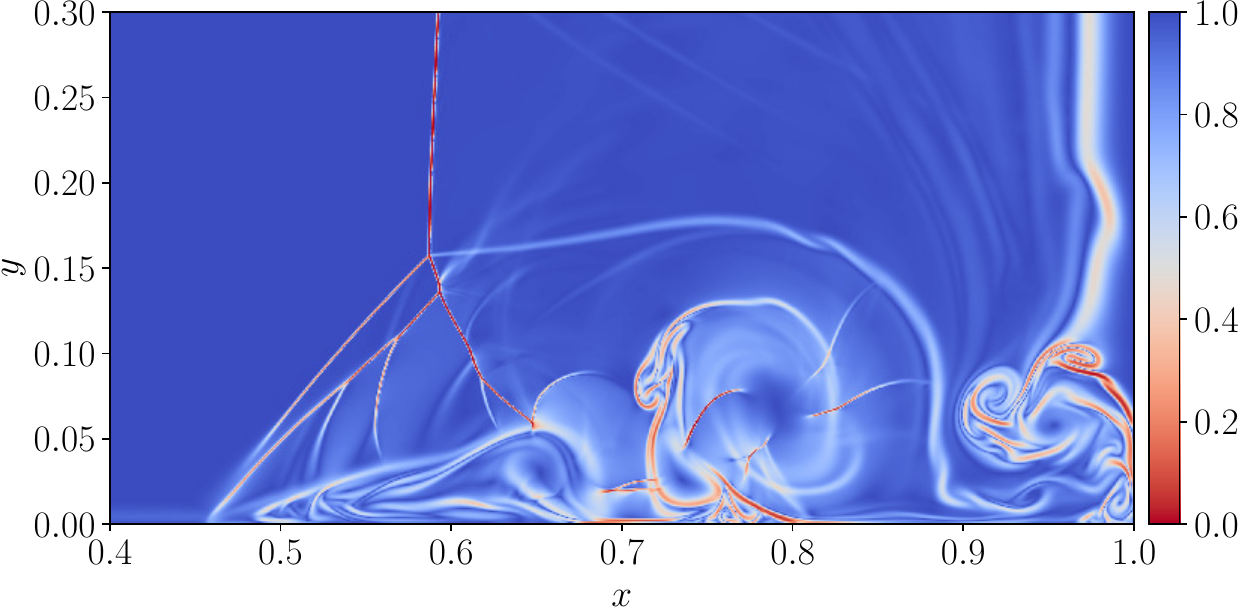}
			\caption{$\phi$ obtained by the AF with $2000\times1000$ cells}
			\label{fig:AF4_fine_mesh}
		\end{subfigure}~
		\begin{subfigure}{0.49\textwidth}
			\includegraphics[width=\textwidth]{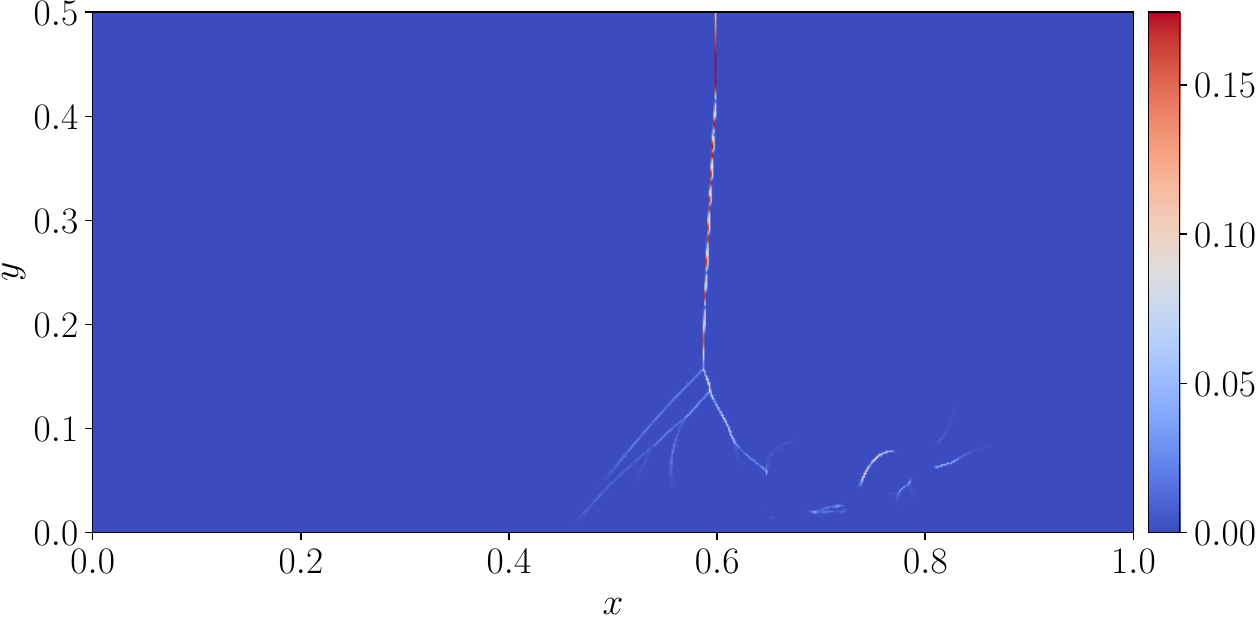}
			\caption{The blending parameter with $2000\times1000$ cells}
		\end{subfigure}
		\caption{Example~\ref{ex:2d_viscous_shock_tube}:  two-dimensional viscous shock tube with shock--boundary-layer interaction with $\mathrm{Re}=1000$.
			All panels correspond to $t_{\texttt{final}}=1$.
			Panels (a)--(e) show a close-up of the region $[0.4,1]\times[0,0.3]$, including the lambda shock, vortices, and shock interaction with the bottom boundary layer, especially the structures in $[0.5,0.7]\times[0,0.05]$.
			Panel (f) shows the limiting activity over the full computational domain $[0,1]\times[0,0.5]$.
		}
		\label{fig:2d_viscous_shock_tube}
	\end{figure}
	
	For a quantitative comparison, the wall density profiles are also presented in Figure~\ref{fig:2d_viscous_shock_tube_wall_density} with the reference data in \cite{Zhou_2018_Grid_PF}.
	The wall density converges with mesh refinement,
	and the AF method matches the reference solution better with the same number of DoFs.
	
	\begin{figure}[htb!]
		\centering
		\includegraphics[width=0.48\textwidth]{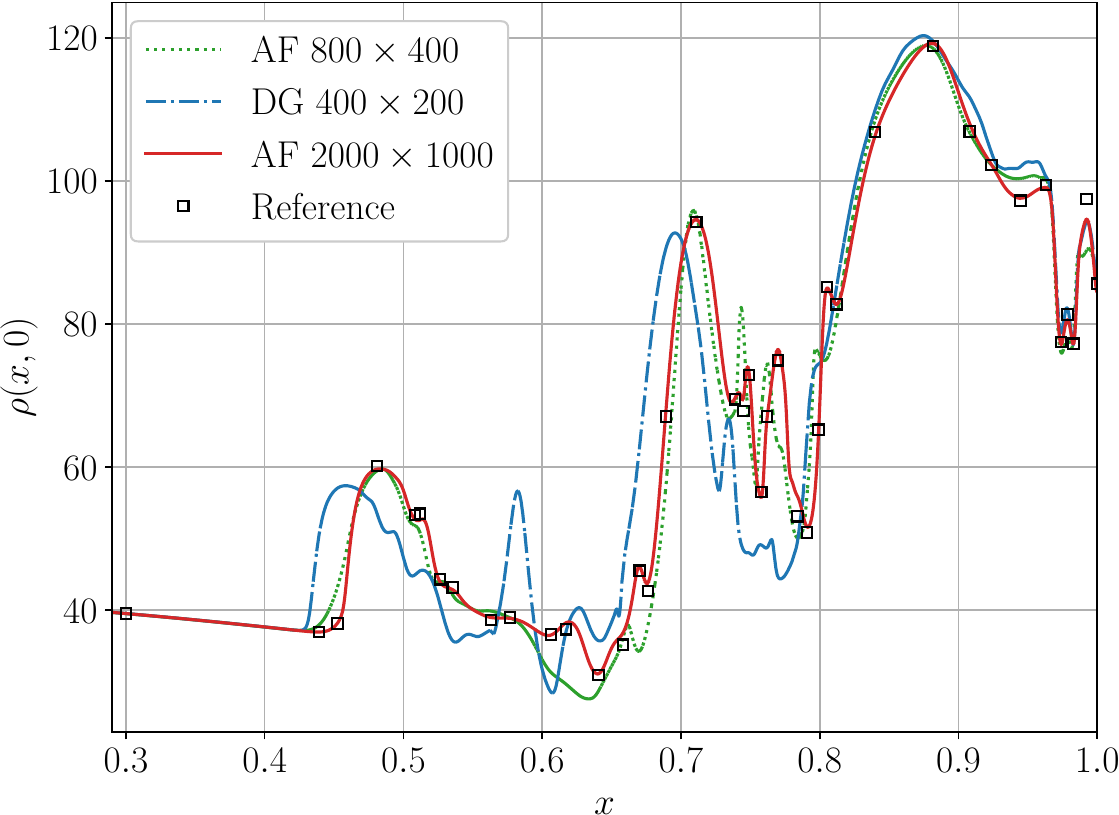}~
		\includegraphics[width=0.48\textwidth]{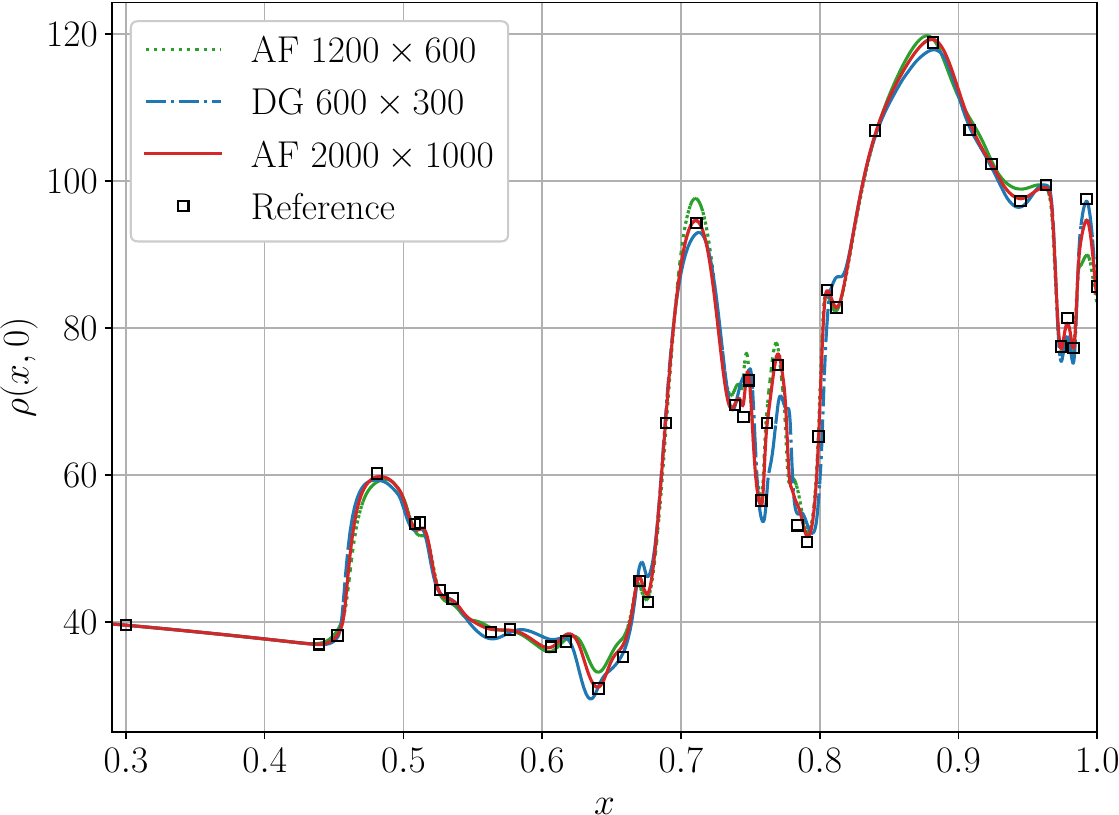}
		\caption{Example~\ref{ex:2d_viscous_shock_tube}: viscous shock tube with $\mathrm{Re}=1000$.
			The density profile on the bottom wall $\rho(x,0)$ compared with the reference data.}
		\label{fig:2d_viscous_shock_tube_wall_density}
	\end{figure}
	
	To see the advantage of the AF method in terms of efficiency,
	the errors of the wall density in the $\ell_1$ norm versus wall-clock time are shown in Figure~\ref{fig:2d_viscous_shock_tube_wall_density_error_time_comparison}.
	Both methods are run using a Mac Studio with M3 Ultra chip with $24$ cores.
	It is clearly seen that the AF method is more efficient with the same number of DoFs,
	and gives better resolution with the same wall-clock time.
	
	\begin{figure}[htb!]
		\centering
		\includegraphics[width=0.48\textwidth]{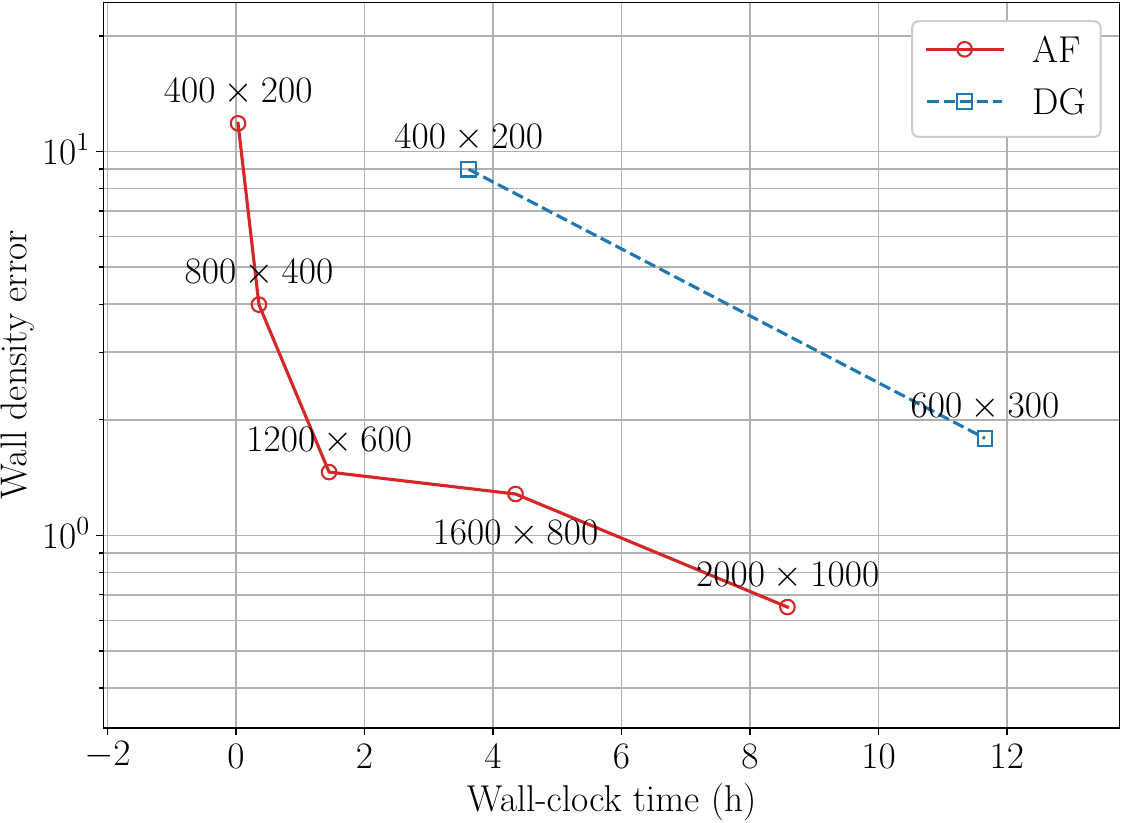}
		\caption{Example~\ref{ex:2d_viscous_shock_tube}: two-dimensional viscous shock tube with shock--boundary-layer interaction with $\mathrm{Re}=1000$.
			The error of the wall density in the $\ell_1$ norm versus the wall-clock time.}
		\label{fig:2d_viscous_shock_tube_wall_density_error_time_comparison}
	\end{figure}
	
	%	\begin{table}[htbp]
		%		\centering
		%		\begin{tabular}{ccr}
			%			\toprule
			%			method & mesh & wall-clock time \\
			%			\midrule
			%			$Q^3$ DG & $400\times200$   & 3h 39min  \\
			%			$Q^3$ DG & $600\times300$   & 11h 01min \\
			%			\midrule
			%			AF      & $800\times400$   & 0h 21min  \\
			%			AF      & $1200\times600$  & 1h 27min  \\
			%			AF      & $2000\times1000$ & 8h 35min  \\
			%			\bottomrule
			%		\end{tabular}
		%		\caption{Example~\ref{ex:2d_viscous_shock_tube}: wall-clock times for the viscous shock tube at $t=1$.}
		%		\label{tab:2d_shock_tube_wall_time}
		%	\end{table}
	
\end{example}

The last two examples consider nearly incompressible viscous flows,
and heat conduction is neglected by setting $\kappa=0$.

\begin{example}[Lid-driven cavity]\label{ex:2d_cavity}
	This classical benchmark describes flow in a square cavity $[0,1]\times[0,1]$ driven by a lid moving with velocity $(1,0)$.
	The initial condition is $(\rho,u,v,p)(x,y,0)=(1,0,0,10^4/\gamma)$,
	and no-slip adiabatic conditions are imposed on all walls. The top lid moves with velocity $(1,0)$, while the other three walls are stationary.
	The reference Mach number is $\mathrm{Ma}=0.01$ so this test falls in the incompressible limit.
	We consider Reynolds numbers $\mathrm{Re}=100$ and $1000$.
	The test is run until $t_{\texttt{final}}=40$ with a $128\times128$ mesh.
	
	The horizontal velocity and streamlines are shown in Figure~\ref{fig:2d_cavity},
	together with the velocity cut-lines along $x=0.5$ and $y=0.5$.
	It is observed that the main recirculation and secondary corner vortices are well captured by the proposed AF method,
	which are comparable to the results in \cite{Tavelli_2017_pressure_JCP},
	and the velocity profiles match the reference solutions given in \cite{Ghia_1982_High_JCP}.
	
	\begin{figure}[htb!]
		\centering
		\includegraphics[width=0.31\textwidth]{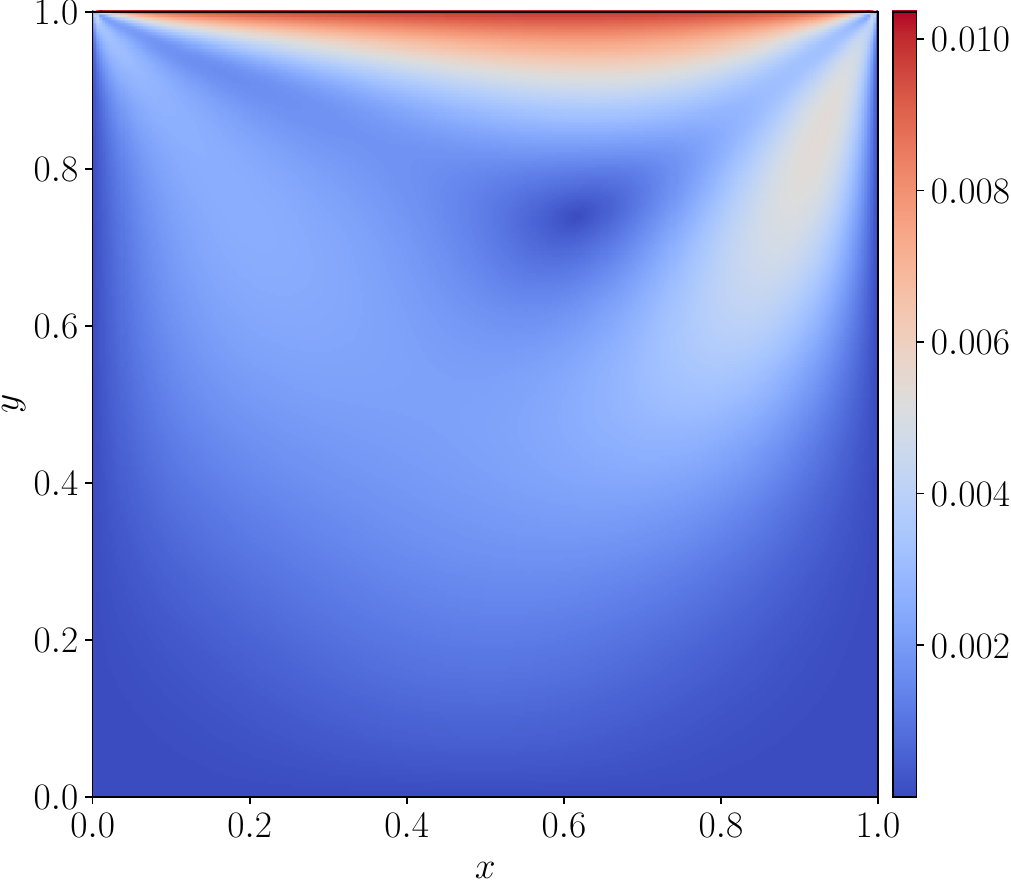}\hfill
		\includegraphics[width=0.31\textwidth]{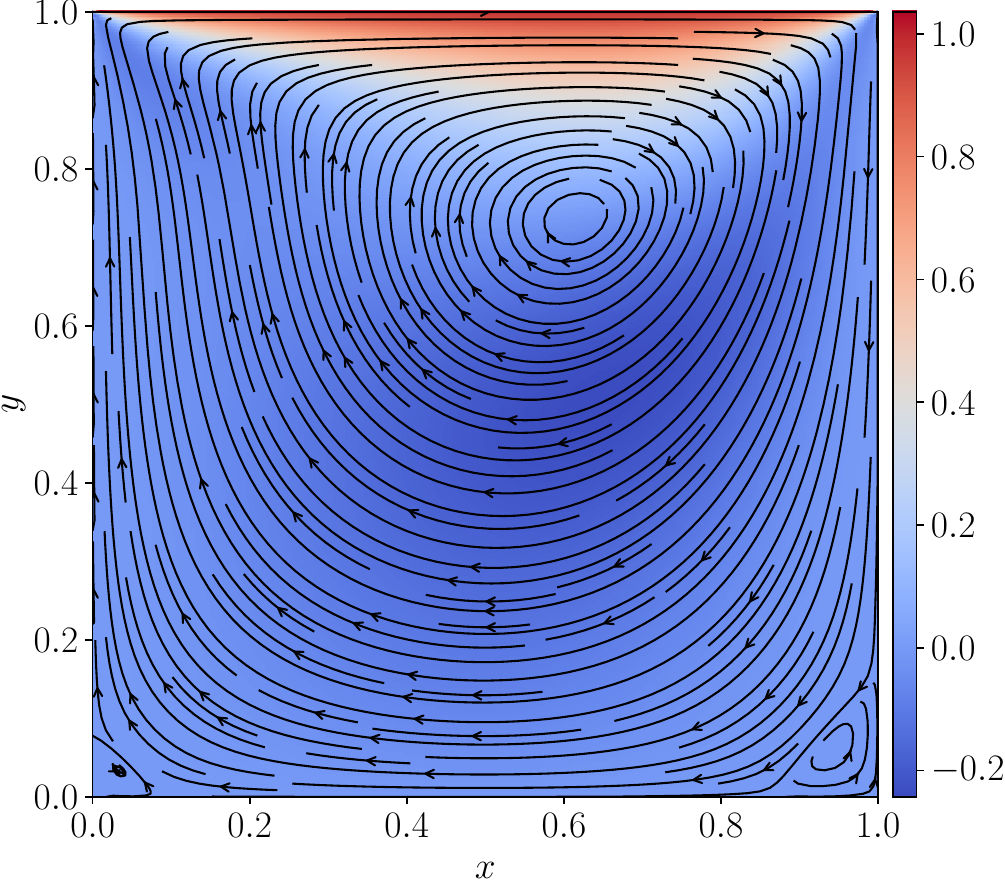}\hfill
		\includegraphics[width=0.355\textwidth]{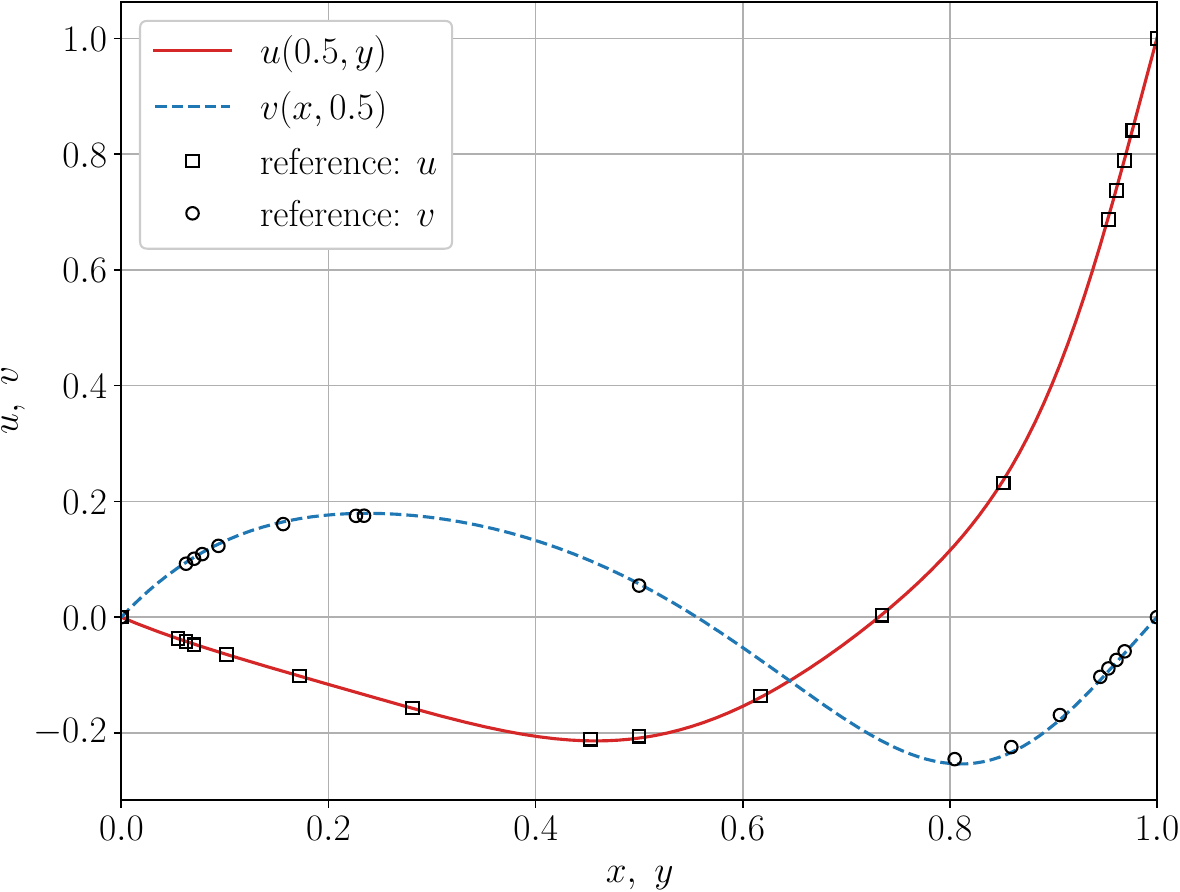}\hfill
		
		\vspace{2pt}
		
		\includegraphics[width=0.31\textwidth]{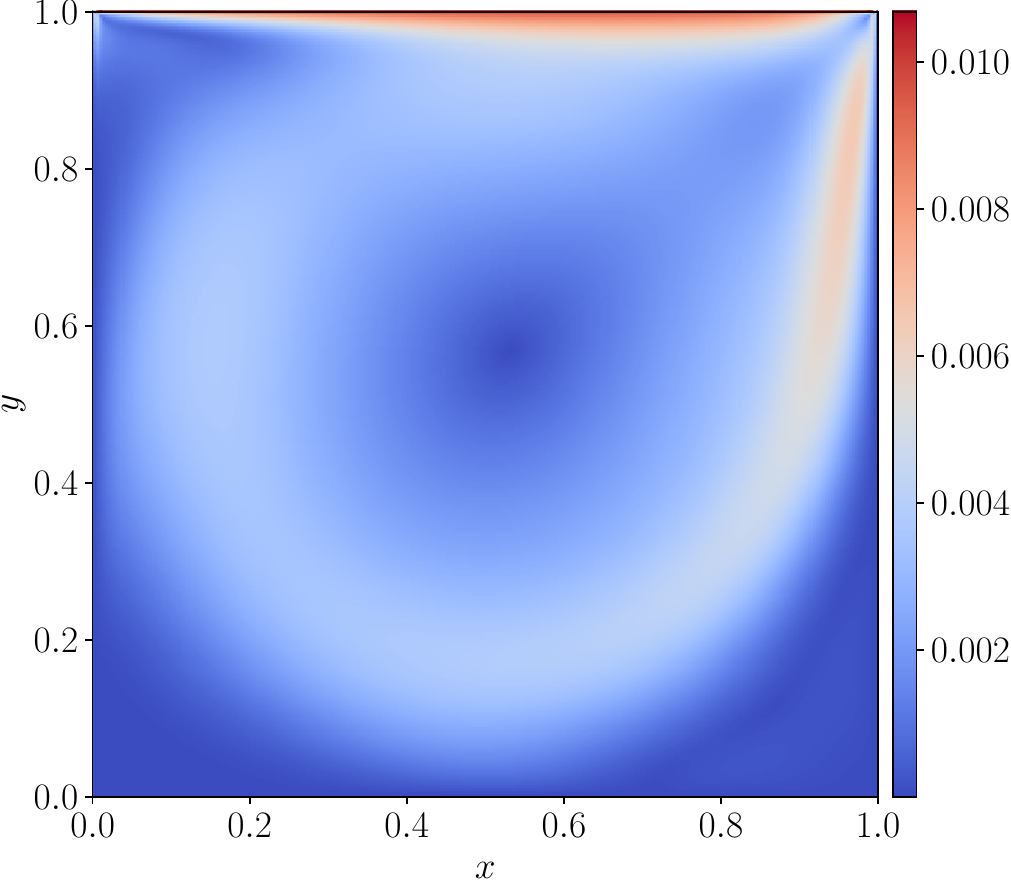}\hfill
		\includegraphics[width=0.31\textwidth]{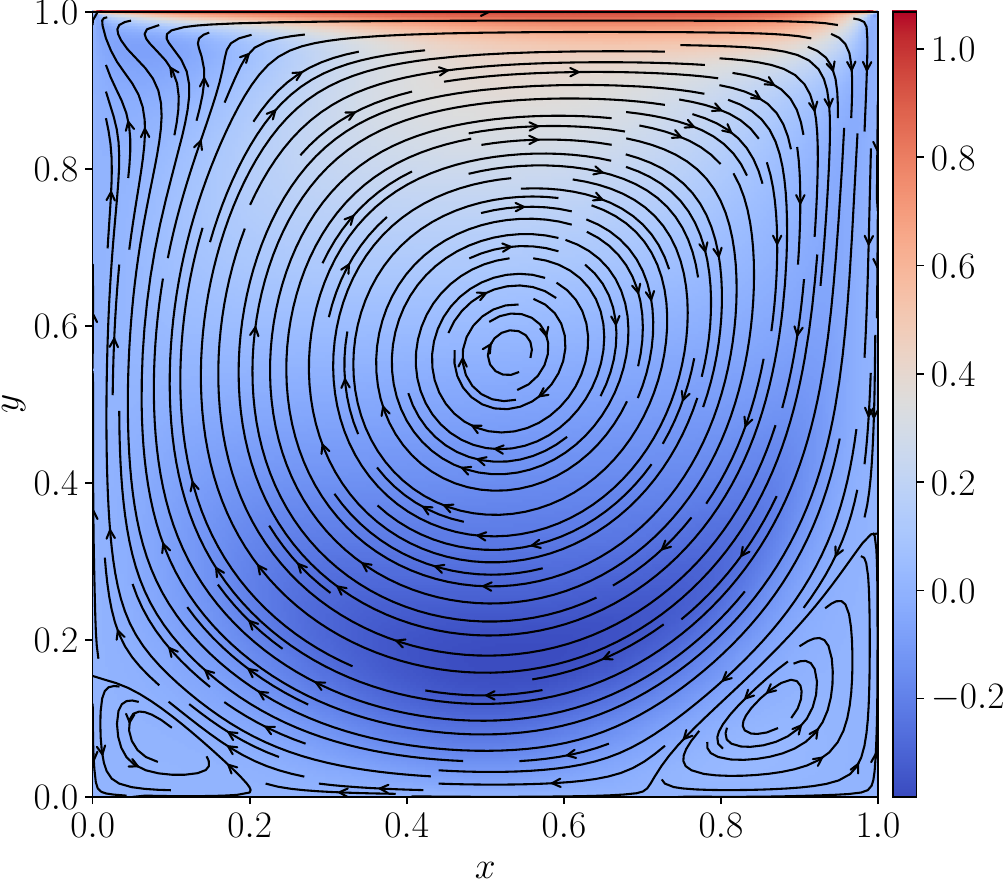}\hfill
		\includegraphics[width=0.355\textwidth]{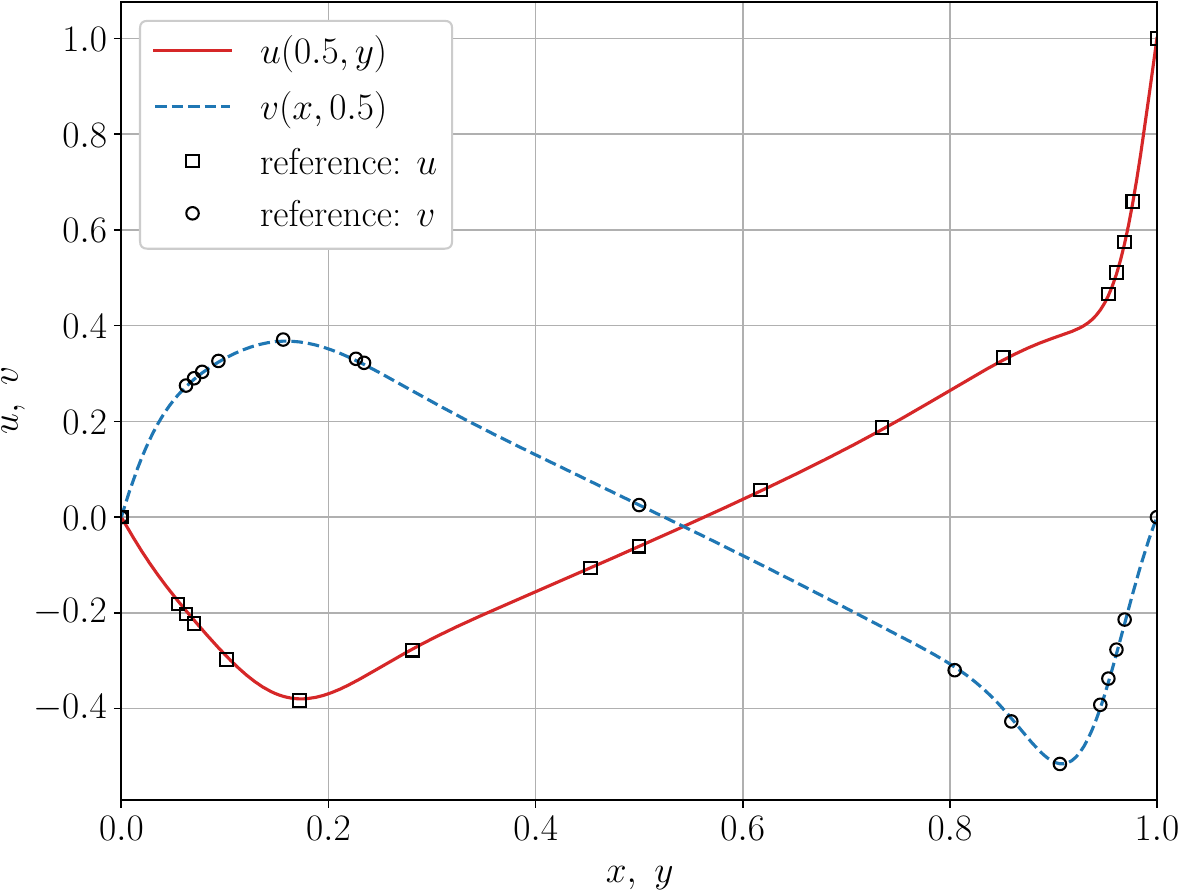}\hfill
		\caption{Lid-driven cavity Example~\ref{ex:2d_cavity} with $\mathrm{Re}=100$ (top) and $1000$ (bottom).
			Left: Mach number, middle: horizontal velocity and streamlines, right: $u(0.5,y)$ and $v(x,0.5)$ compared with the reference solutions.}
		\label{fig:2d_cavity}
	\end{figure}
	
\end{example}

\begin{example}[Double shear layer]\label{ex:2d_double_shear_layer}
	This test considers the evolution of thin shear layers at low Mach numbers,
	following the setup in \cite{Boscheri_2025_all_CMAME}.
	The computational domain is $[0,1]\times[0,1]$ with periodic boundary conditions.
	The initial density and	pressure are $1$ and $10^4/\gamma$, and the velocity is
	\begin{equation*}
		\begin{aligned}
			u&=
			\begin{cases}
				\tanh(30(y-0.25)), &\text{if}~ y\leqslant 0.5,\\
				\tanh(30(0.75-y)), &\text{otherwise},
			\end{cases}\\
			v&=0.05\sin(2\pi x).
		\end{aligned}
	\end{equation*}
	The dynamic viscosity is taken as $\mu=2\times10^{-4}$ and the test is solved until $t_{\texttt{final}}=1.8$.
	The Mach number is $\mathrm{Ma}=0.01$ and the Reynolds number is $\mathrm{Re}=10000$.
	
	The vorticity $\partial_xv-\partial_yu$ at $t=0.4$, $0.8$, $1.2$, $1.8$ obtained with $200\times200$ cells is shown in Figure~\ref{fig:2d_double_shear_layer},
	which agrees well with the results in \cite{Boscheri_2025_all_CMAME},
	thus the proposed AF method seems to work well in the low-Mach regime.
	
	\begin{figure}[htb!]
		\centering
		\includegraphics[width=0.48\textwidth]{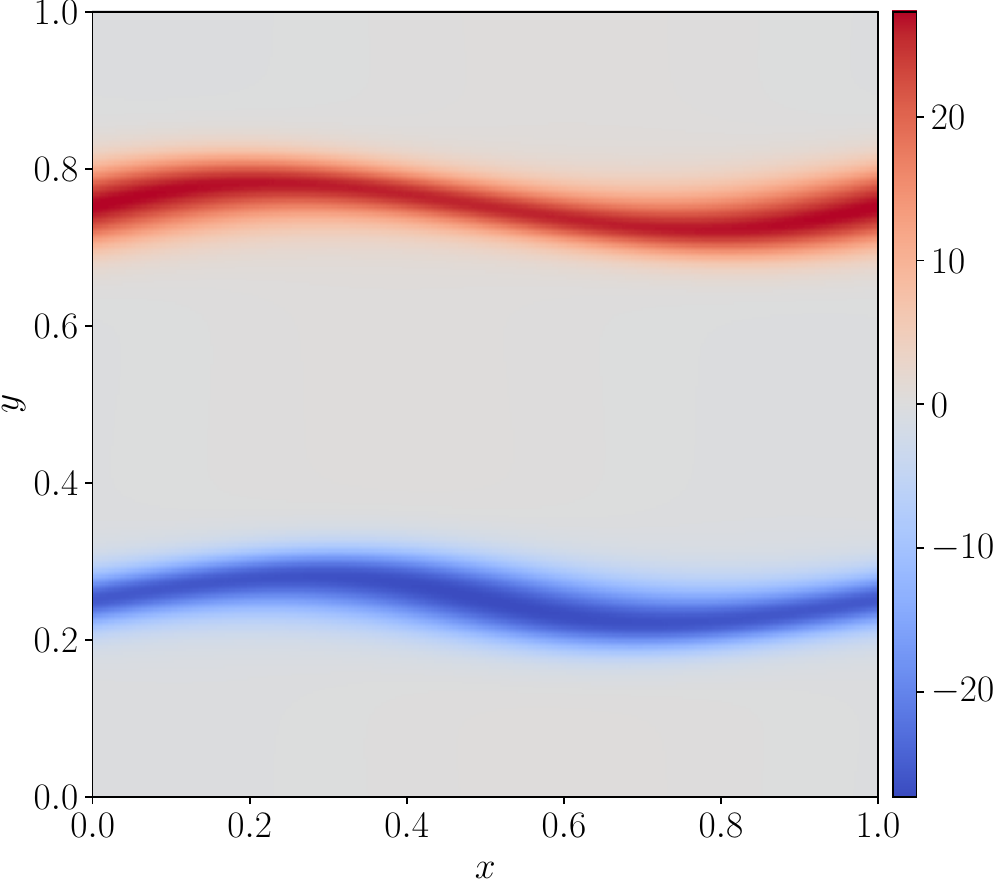}\hfill
		\includegraphics[width=0.48\textwidth]{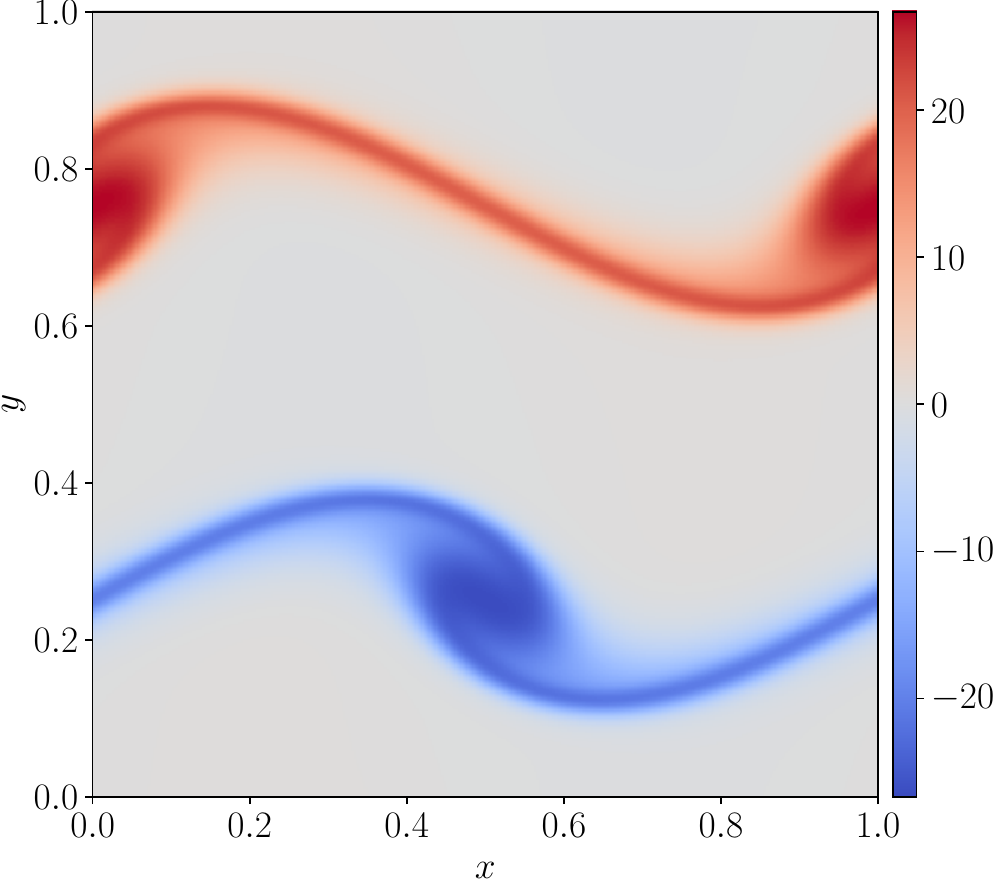}\hfill
		
		\vspace{2pt}
		
		\includegraphics[width=0.48\textwidth]{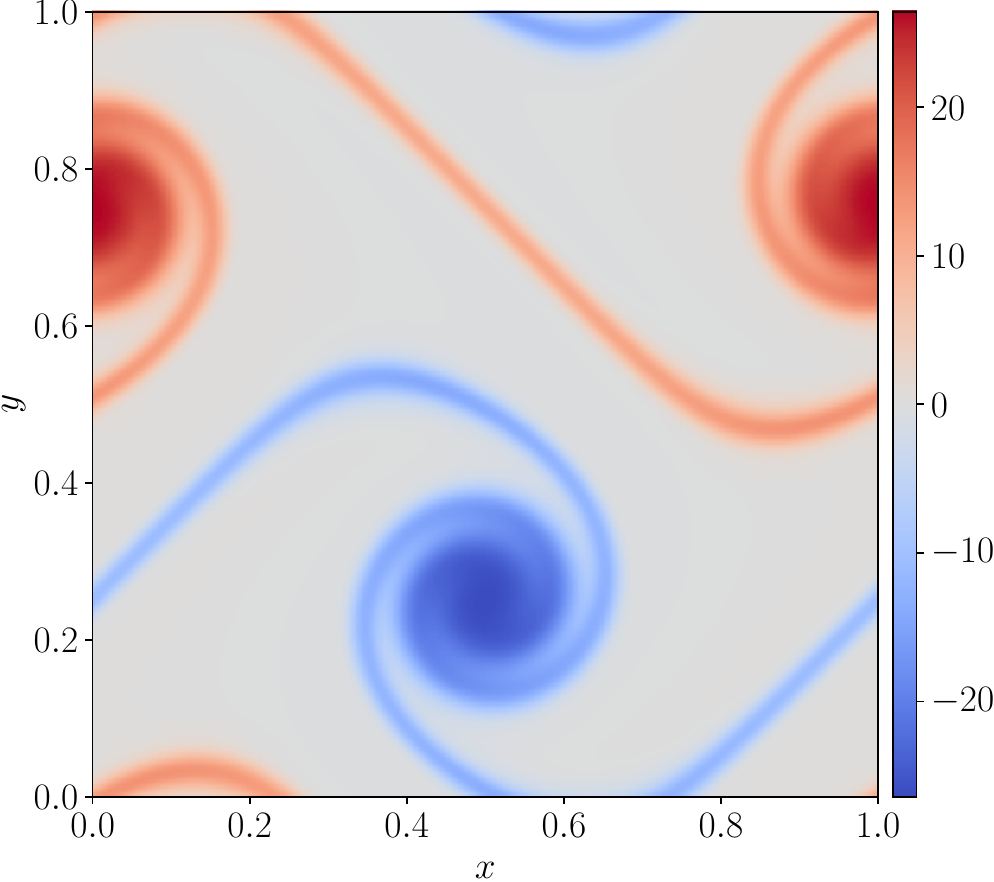}\hfill
		\includegraphics[width=0.48\textwidth]{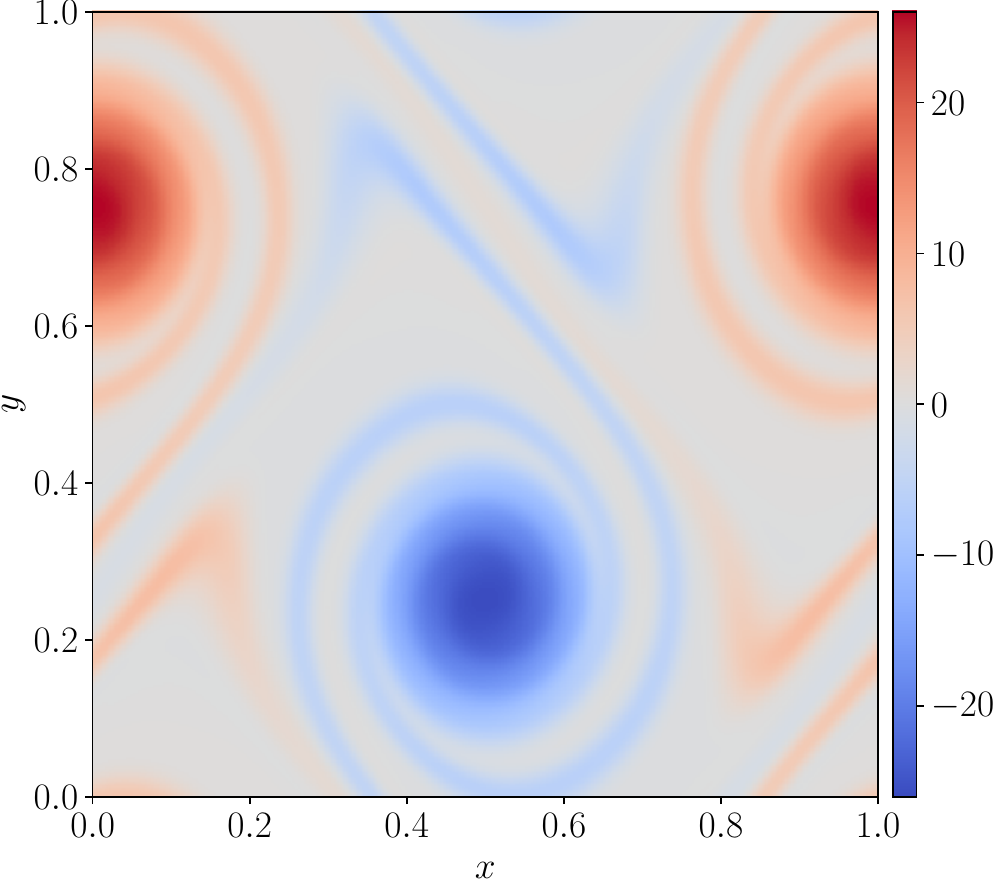}\hfill
		\caption{The vorticity of Example~\ref{ex:2d_double_shear_layer} with $\mathrm{Re}=10000$
			at $t=0.4$, $0.8$, $1.2$, $1.8$ (from top left to bottom right).}
		\label{fig:2d_double_shear_layer}
	\end{figure}
\end{example}

\section{Conclusions}\label{sec:conclusion}

We have developed a compact fourth-order positivity-preserving AF method for the one- and two-dimensional compressible Navier--Stokes equations on Cartesian meshes.
The viscous flux divergence is discretized directly using compact fourth-order operators.
The inviscid discretization achieves fourth-order accuracy while retaining the degrees of freedom of the standard third-order AF method.
The CFL condition is also less restrictive than that of the standard third-order AF method due to the additional downwind point included in the biased stencil.
Monolithic flux limiting treats the inviscid and viscous fluxes jointly in the cell-average update and maintains local conservation.
Together with a scaling limiter for point values, it preserves density and pressure positivity and suppresses nonphysical oscillations near strong discontinuities.
The numerical experiments demonstrate fourth-order convergence for smooth viscous flows and positivity preservation in problems involving strong discontinuities.
In the two-dimensional viscous shock tube with shock--boundary-layer interaction, comparisons with the DG method show that AF attains comparable accuracy with fewer degrees of freedom and less wall-clock time.

\section*{Acknowledgements}
The first author was partially supported by National Natural Science Foundation of China (No. 12671481), 1+1+1 CUHK-CUHK(SZ)-GDSTC Joint Collaboration Fund (No. 2025A0505000079), University Development Fund (No. UDF01004388), and the Alexander von Humboldt Foundation Research Fellowship (No. CHN-1234352-HFST-P).

% Appendix disabled for this version.
%\appendix
%\input{Appendix}

\end{document}